\documentclass[final]{siamart190516}
\usepackage{amsfonts}
\usepackage{amsfonts,amsmath,amssymb}
\usepackage{mathrsfs,mathtools,stmaryrd,wasysym}
\usepackage{enumerate}
\usepackage{enumitem}
\usepackage{comment}  
\usepackage{esint}
\usepackage{graphicx,psfrag}
\usepackage{hyperref}
\usepackage{amsmath,stackengine}
\usepackage{yfonts}
\usepackage{enumitem}
\DeclareMathAlphabet{\mathpzc}{OT1}{pzc}{m}{it}

\newcommand{\EO}[1]{{\color{black}#1}}
\newcommand{\DQ}[1]{{\color{black}#1}}
\newcommand{\MS}[1]{{\color{black}#1}}

\newsiamremark{remark}{Remark}

\newcommand{\TheTitle}{A bilinear pointwise tracking optimal control problem for a semilinear elliptic PDE}
\newcommand{\ShortTitle}{Bilinear pointwise tracking optimal control}
\newcommand{\TheAuthors}{E. Ot\'arola, D. Quero, M. Sasso}

\headers{\ShortTitle}{\TheAuthors}

\title{{\TheTitle}\thanks{EO is partially supported by ANID through FONDECYT grant 1220156. MS is supported by \MS{Direcci\'on de Postgrado of} UTFSM through Programa de Incentivos a la Investigaci\'on Cient\'ifica (PIIC) \MS{No. 005/2024} and by UV through Becas FIB\MS{-}UV \MS{2025}. }} 

\author{Enrique Ot\'arola\thanks{Departamento de Matem\'atica, Universidad T\'ecnica Federico Santa Mar\'ia, Valpara\'iso, Chile. \email{enrique.otarola@usm.cl}}
\and
Daniel Quero\thanks{Departamento de Matem\'atica, Universidad T\'ecnica Federico Santa Mar\'ia, Valpara\'iso, Chile \email{daniel.quero@alumnos.usm.cl}}
\and
Mat\'ias Sasso\thanks{Departamento de Matem\'atica, Universidad T\'ecnica Federico Santa Mar\'ia, Valpara\'iso, Chile \email{matias.sasso@sansano.usm.cl}}}

\ifpdf
\hypersetup{
  pdftitle={\TheTitle},
  pdfauthor={\TheAuthors}
}
\fi

\date{Draft version of \today.}

\begin{document}

\maketitle
\begin{abstract}
We consider a bilinear optimal control problem with pointwise tracking for a semilinear elliptic PDE in two and three dimensions. The control variable enters the PDE as a (reaction) coefficient and the cost functional contains point evaluations of the state variable. These point evaluations lead to an adjoint problem with a linear combination of Dirac measures as a forcing term. In Lipschitz domains, we derive the existence of optimal solutions and analyze first and  necessary and sufficient second order optimality conditions. \EO{We also prove that every locally optimal control $\bar u$ belongs to $H^1(\Omega)$. Finally, assuming that the domain $\Omega \subset \mathbb{R}^2$ is \DQ{a} convex \DQ{polygon}, we prove that $\bar u \in C^{0,1}(\bar \Omega)$.}
\end{abstract}


\section{Introduction}
\label{sec:intro}
The aim of this work is to study optimality conditions and regularity estimates for an optimal control problem with pointwise tracking. The state equation corresponds to a semilinear elliptic partial differential equation (PDE) and the control variable enters the state equation as a reaction coefficient; constraints on the control variable are also considered. To make the discussion more precise, let $d \in \{2,3\}$, let $\Omega \subset \mathbb{R}^{d}$ be an open and bounded domain with Lipschitz boundary $\partial \Omega$, and let $\mathcal{D}$ be a finite ordered subset of $\Omega$. Given a set of desired states $\{y_{t}\}_{t \in \mathcal{D}}$ and a regularization parameter $\alpha > 0$, we introduce the cost functional
\begin{equation}\label{def:cost_func}
J(y,u) := \dfrac{1}{2}\sum_{t \in \mathcal{D}}(y(t) - y_{t})^{2} + \dfrac{\alpha}{2}\|u\|^{2}_{L^{2}(\Omega)}.
\end{equation}
Let $f \in H^{-1}(\Omega)$ be a fixed function. We are then interested in the following optimal control problem: Find $\min J(y,u)$ subject to the \emph{semilinear elliptic} PDE
\begin{equation}\label{def:state_eq}
-\Delta y + a(\cdot,y) + uy = f
\
\text{in } \Omega,
\qquad y = 0
\
\text{on } \partial \Omega,
\end{equation}
and the \emph{control constraints}
\begin{equation}\label{def:box_const}
u \in \mathbb{U}_{ad},
\qquad
\mathbb{U}_{ad}:= \{v \in L^{2}(\Omega): \mathtt{a} \leq v(x) \leq \mathtt{b} \text{ for a.e.}~x \in \Omega\}.
\end{equation}
The control bounds $\mathtt{a}$ and $\mathtt{b}$ are real numbers such that $-\infty < \mathtt{a} < \mathtt{b} < \infty$ and the nonlinear function $a: \Omega \times \mathbb{R} \rightarrow \mathbb{R}$ is a suitable Carath\'eodory function that satisfies the structural assumptions stated in section \ref{sec:assumptions} below.

The optimal control of systems governed by PDEs has important applications in various scientific fields \cite{MR2516528,MR271512,MR2583281,MR2839219}. In general, the mathematical modeling involves a control that appears as an additive term in the system, which is called \emph{additive control}. Based on the seminal work of J. L. Lions \cite{MR271512}, considerable efforts have been made in recent decades to study additive optimal control, and numerous mathematical and computational tools have been developed; see, for instance, \cite{MR2349487,MR2516528,MR2583281,MR3308473,MR4599941}. However, additive controls are not able to modify some of the key intrinsic properties of certain systems \cite{MR0414174,MR2641453}. For example, if we want to change the reaction rate in some chain reaction--like processes from chemistry, additive controls amount to controlling by adding or removing a certain amount of the reactants, which is not realistic. To solve this problem, it is useful to use specific catalysts to control the systems, which can be mathematically modeled by \emph{bilinear optimal control} \cite{MR0414174,MR2641453}.

In contrast to the additive case, the bilinear counterpart, also known as \emph{multiplicative control}, enters the state equation as a coefficient that interacts multiplicatively with the state variable. Several applications of bilinear optimal control can be found in the literature. In medicine, for example, bilinear controls can be used in modeling cancer treatments such as chemotherapy, where the drug dose (control) interacts with the cancer cell population (state) \cite{ledzewicz2004,Zerrik_etal,SERHAL2025104362}. Other applications include neutron transport in physics \cite{MR735811} and ecosystem management in ecology \cite{MR0332249}. 

Mathematically, bilinear optimal control can be formulated as in \eqref{def:state_eq}: The control variable $u$ enters the state equation as a coefficient and generates the nonlinear coupling $uy$, which exactly shows the bilinear structure of the control \cite{MR2536007,Casas_bilinear_1}. Compared to the additive case \cite{MR2583281}, the bilinear optimal control entails some additional difficulties. First, even in the linear case $a \equiv 0$ in \eqref{def:state_eq}, the control\DQ{-}to-state map is nonlinear, so the uniqueness of optimal solutions cannot be guaranteed \cite{MR2536007,Casas_bilinear_1}. As a result, the optimization problem is nonconvex and a complete optimization study requires the derivation of second order optimality conditions \cite{MR2536007,Casas_bilinear_1,Casas_bilinear_2}. Second, the sign of $u$ is not necessarily positive. Therefore, the derivation of the second-order conditions differs significantly from the classical case \cite{MR2583281,MR2902693,MR3586845}: several results concerning the well-posedness and differentiability properties of the control-to-state and \EO{control-to-adjoint state} maps are not standard \cite{Casas_bilinear_1}. We refer the reader to \cite{MR2536007,MR4416986,MR4679130,Casas_bilinear_1} for the analysis of some bilinear optimal control problems.

Our optimal control problem \eqref{def:cost_func}--\eqref{def:box_const} corresponds to a special case of bilinear optimal control. In particular, \eqref{def:cost_func}--\eqref{def:box_const} is characterized by a cost functional $J$ that contains point evaluations of the state variable. This structure leads to an adjoint problem with a linear combination of Dirac measures as a forcing term. As a consequence, and in contrast to the state equation, which can be posed naturally in $H^{1}_{0}(\Omega)$, the analysis of the adjoint problem must be performed in a less regular Sobolev space, for example, in $W^{1,r}_{0}(\Omega)$ ($r <  d/(d-1)$) as in \cite{MR0812624,MR3449612,MR3973329,MR4438718}. The latter complicates the analysis of the optimal control problem, especially when deriving optimality conditions and regularity estimates.

Apart from the fact that this presentation is the first one that studies a bilinear optimal control problem with pointwise tracking, the analysis itself entails a number of difficulties. To overcome them, we had to provide several results. Let us briefly discuss some of them:
\begin{itemize}[leftmargin=*]
\item \emph{The state equation}: We review the well-posedness of \eqref{def:state_eq} for controls $u$ that belong to the set $\mathcal{A}_0 \subset L^2(\Omega)$ defined in \eqref{eq:mathcal_A_0}. Note that the involved controls are not necessarily positive. We also derive regularity results and analyze differentiability properties of the underlying control-to-state map.
\item \emph{The adjoint equation}: We derive the well-posedness of the \emph{singular} adjoint problem for controls $u$ belonging to the set $\mathcal{A} \subset L^2(\Omega)$ defined in \eqref{eq:mod_mathcal_A}. We also analyze the differentiability properties of the underlying control-to-adjoint state map. Finally, under the assumption that $\Omega$ is \MS{a convex polytope}, we prove that the solution of the adjoint equation belongs to $H^2$ and $C^{0,1}$ away from the singular points.
\item \emph{Optimality conditions:} We derive first order and second order necessary and sufficient optimality conditions. To this end, we have adapted the arguments from the additive case to the bilinear scenario and also dealt with the non-smoothness of the corresponding adjoint state.
\item \emph{Regularity of locally optimal controls:} \MS{We prove that every locally optimal control $\bar{u}$ belongs to $H^1(\Omega)$. If $\MS{\Omega \subset \mathbb{R}^{2}}$ is a convex polygon, we prove that $\bar{u} \in C^{0,1}(\bar \Omega)$. This requires a precise understanding of the multiplication $\bar{y} \bar{p}$ near the singular points.} 
\end{itemize}
The structure of this manuscript is as follows. Section \ref{sec:notation_and_prel} establishes the main notation and assumptions that we use in this paper. In section \ref{sec:semilinear_eq}, we prove the existence and uniqueness of solutions for a weak formulation of the state equation \eqref{def:state_eq}, derive suitable regularity properties for the state variable, and analyze differentiability properties of the corresponding mapping $u \mapsto y$. In section \ref{sec:OCP}, we propose and analyze a weak formulation for the optimal control problem \eqref{def:cost_func}--\eqref{def:box_const}. More precisely, we show the existence of optimal solutions, analyze the adjoint problem and derive regularity estimates for its solution as well as first and second order optimality conditions. We conclude in section \ref{sec:loc_opt_control_reg} with regularity results for locally optimal controls.


\section{Notation and assumptions}
\label{sec:notation_and_prel}

In this section we present the main notation and assumptions under which we will work.


\subsection{Notation}\label{sec:notation}
 In the context of our work, $d\in \{2,3\}$ and $\Omega \subset \mathbb{R}^d$ is an open and bounded domain with Lipschitz boundary $\partial \Omega$. We will make further regularity requirements on $\Omega$ to perform regularity estimates.

For a Banach function space $\mathfrak{X}$, we denote the dual and the norm of $\mathfrak{X}$ by $\mathfrak{X}'$ and $\|\cdot\|_{\mathfrak{X}}$, respectively. Let $\mathfrak{Y}$ be another Banach function space. We write $\mathfrak{X}\hookrightarrow\mathfrak{Y}$ to denote that $\mathfrak{X}$ is continuously embedded in $\mathfrak{Y}$. We denote by $\langle \cdot, \cdot\rangle_{\mathfrak{X}',\mathfrak{X}}$ the duality pairing between $\mathfrak{X}'$ and $\mathfrak{X}$. If $\mathfrak{X}'$ and $\mathfrak{X}$ are clear from the context, we write $\langle \cdot, \cdot\rangle$. Let $\{ x_n \}_{n\in \mathbb{N}} \subset \mathfrak{X}$. We denote the strong, the weak, and the weak$^{\star}$ convergence of $\{ x_n \}_{n\in \mathbb{N}}$ to $x$ in $\mathfrak{X}$ as $n \uparrow \infty$ by $x_n \rightarrow x$, $x_n \rightharpoonup x$, and $x_{n} \mathrel{\ensurestackMath{\stackon[1pt]{\rightharpoonup}{\scriptstyle\ast}}} x$, respectively.

By $a \lesssim b$ we mean $a \leq C b$, with a constant $C>0$ that depends neither on $a$ nor on $b$. The value of $C$ might change at each occurrence. If the particular value of $C$ is important for our analysis, we will give it a name.

\subsection{Assumptions}\label{sec:assumptions}

To perform an analysis for our optimal control problem, we make three assumptions on $a$: \ref{A1}, \ref{A2}, and \ref{A3}. We emphasize that assumptions \ref{A1} and \ref{A3} are classical assumptions in the analysis of optimal control problems governed by semilinear elliptic PDEs \cite{inbook,MR2583281}. In contrast, the assumption \ref{A2} has recently been introduced in \cite{Casas_bilinear_1} in the context of bilinear optimal control.

\begin{enumerate}[label=(A.\arabic*)]
\item \label{A1} $a : \Omega \times \mathbb{R} \rightarrow \mathbb{R}$ is a Carath\'eodory function of class $C^{2}$ with respect to the second variable and $a(\cdot , 0) \in L^{2}(\Omega)$.
\item \label{A2} There exists a function $a_{0} \in L^{\infty}(\Omega)$ such that $\frac{\partial a}{\partial y}(x,y) \geq a_{0}(x)$ for a.e.~$x \in \Omega$ and for all $y \in \mathbb{R}$.
\item \label{A3} For all $\mathfrak{m}>0$, there exists a positive constant $C_{a,\mathfrak{m}}$ such that
\begin{equation*}
\sum_{i=1}^{2} \left| \frac{\partial^{i} a}{\partial y^{i}}(x,y) \right| \leq C_{a,\mathfrak{m}},
\qquad
\left| \frac{\partial^{2} a}{\partial y^{2}}(x,v) - \frac{\partial^{2} a}{\partial y^{2}}(x,w) \right| \leq C_{a,\mathfrak{m}}\left| v-w\right|
\end{equation*}
for a.e.~$x \in \Omega$ and $y,v,w \in [-\mathfrak{m},\mathfrak{m}]$.
\end{enumerate}


\section{The state equation}\label{sec:semilinear_eq}

In this section we analyze the state equation \eqref{def:state_eq}. For this purpose, following the ideas developed in \cite{Casas_bilinear_1}, we introduce the set
\begin{equation}
\mathcal{A}_{0} := \{ u \in L^{2}(\Omega): a_{0}(x)+ u(x) \geq 0~\text{for a.e.}~x\in \Omega \},
\label{eq:mathcal_A_0}
\end{equation}
where $a_0$ is as in \ref{A2}. Given $f \in L^{q}(\Omega)$ for some $q > d/2$, we introduce the following weak formulation of the state equation \eqref{def:state_eq}: Find $y \in H_{0}^{1}(\Omega)$ such that
\begin{equation}\label{eq:weak_semi_bilinear_eq}
\int_{\Omega} \nabla y \cdot \nabla v \mathrm{d}x + \int_{\Omega} a(\cdot,y) v \mathrm{d}x 
+ \int_{\Omega} u y v \mathrm{d}x = \int_{\Omega} fv \mathrm{d}x
    \quad \forall v \in H_{0}^{1}(\Omega).
\end{equation} 

The well-posedness of the weak problem \eqref{eq:weak_semi_bilinear_eq} is as follows. We note that, contrary to the usual assumptions $u \in L^{\infty}(\Omega)$ and $u \geq 0$ a.e.~in $\Omega$ (see, for instance, \cite[Assumption 4.2 (ii)]{MR2583281}), we have that $u \in L^2(\Omega)$ and that $a_0 + u \geq 0$ a.e.~in $\Omega$.

\begin{theorem}[well-posedness]
Let us assume that \ref{A1}--\ref{A3} hold. Given $u \in \mathcal{A}_0$ and $f \in L^{q}(\Omega)$ for some $q > d/2$, there exists a unique solution $y \in H_0^1(\Omega) \cap L^{\infty}(\Omega)$ for problem \eqref{eq:weak_semi_bilinear_eq}. In addition, we have the following stability bounds
\begin{equation}
\label{eq:stability_weak_problem}
\| \nabla y\|_{L^{2}(\Omega)} \leq \| f-a(\cdot,0) \|_{H^{-1}(\Omega)},
\qquad
\|y\|_{L^{\infty}(\Omega)} \lesssim \| f-a(\cdot,0) \|_{L^{q}(\Omega)},
\end{equation}
where the hidden constant in the $L^{\infty}(\Omega)$-estimate is independent of $y$, $a$, and $f$.
\label{thm:well-posedness-semilinearPDE}
\end{theorem}
\begin{proof}
We begin the proof by noting that the assumptions \ref{A1}--\ref{A3} guarantee the following property: for all $\mathfrak{m}>0$, $|a(x,y)| \leq C_{a,\mathfrak{m}} \mathfrak{m} + |a(x,0)| =: \psi_{\mathfrak{m}}(x)$ for a.e.~$x \in \Omega$ and for all $|y| \leq \mathfrak{m}$. Note that for all $\mathfrak{m}>0$,  $\psi_{\mathfrak{m}}$ belongs to $L^q(\Omega)$ for some $q>d/2$. Next, as in the proof of \cite[Theorem 2.4]{Casas_bilinear_1}, we introduce the function
\[
 b : \Omega \times \mathbb{R} \rightarrow \mathbb{R}:
 \qquad
 (x,y) \mapsto b(x,y):= a(x,y) - a(x,0) - a_0(x)y.
\]
The function $b$ satisfies the following three properties. First, $b(x,0) = 0$ for a.e.~$x \in \Omega$. Second, $\partial b/ \partial y (x,y) = \partial a/ \partial y (x,y) - a_0(x) \geq 0$ for a.e.~$x \in \Omega$ and for all $y \in \mathbb{R}$. Third, for all $\mathfrak{m}>0$, $|b(x,y)| \leq ( C_{a,\mathfrak{m}} + |a_0(x)|) \mathfrak{m} =: \chi_{\mathfrak{m}}(x)$ for a.e.~$x\in \Omega$ and for all $|y| \leq \mathfrak{m}$. For all $\mathfrak{m}>0$, the function $\chi_{\mathfrak{m}}$ belongs to $L^{\infty}(\Omega)$. With the function $b$ in hand, we rewrite the problem \eqref{eq:weak_semi_bilinear_eq} as follows: Find $y \in H_0^1(\Omega)$ such that, for every $v \in H_0^1(\Omega)$,
\begin{equation}\label{eq:weak_semi_bilinear_eq_b}
    \int_{\Omega}\nabla y  \cdot \nabla v \mathrm{d}x + \int_{\Omega} b(\cdot,y)v \mathrm{d}x  + \int_{\Omega}[u + a_0] yv \mathrm{d}x = \int_{\Omega}[f - a(\cdot,0)]v \mathrm{d}x.
\end{equation}
Let us now analyze the problem \eqref{eq:weak_semi_bilinear_eq_b}. For this purpose, we introduce a truncation $b_k$ of the function $b$ as follows: For an arbitrary $k>0$, we define, for a.e.~$x \in \Omega$,
\[
b_k(x,y) = b(x,k) 
\textrm{ if } y > k,
\, \,
b_k(x,y) = b(x,y) 
\textrm{ if } |y| \leq k,
\, \,
b_k(x,y) = b(x,-k) 
\textrm{ if } y < - k.
\]
We note that $|b_k(x,y)| \leq \chi_k(x)$ for a.e.~$x \in \Omega$ and for all $y \in \mathbb{R}$, where $\chi_k \in L^{\infty}(\Omega)$. With the truncation $b_{k}$ at hand, we introduce the following weak problem: Find $y_k \in H_0^1(\Omega)$ such that, for every $v \in H_0^1(\Omega)$,
\begin{equation}\label{eq:weak_semi_bilinear_eq_b_k}
    \int_{\Omega}\nabla y_k \cdot \nabla v \mathrm{d}x + \int_{\Omega}b_k(\cdot,y_k)v \mathrm{d}x + \int_{\Omega}[u + a_0] y_kv \mathrm{d}x = \int_{\Omega}[f - a(\cdot,0)]v \mathrm{d}x.
\end{equation}
Define $\mathfrak{A}: H_0^1(\Omega) \rightarrow H^{-1}(\Omega)$ as
$
 \langle \mathfrak{A}(w),v \rangle :=  \int_{\Omega} \nabla w \cdot \nabla v \mathrm{d}x + \int_{\Omega} (u + a_0) w v \mathrm{d}x
$
for every $v \in H_0^1(\Omega)$. The linear map $\mathfrak{A}$, which is related to the linear part of \eqref{eq:weak_semi_bilinear_eq_b_k}, is continuous and coercive in $H_0^1(\Omega)$. In fact, for every 
$w \in H_0^1(\Omega)$, we have
\begin{equation*}
\| \mathfrak{A}(w) \|_{H^{-1}(\Omega)}
\leq
\left[ 1
+
C_{4 \hookrightarrow 2}^2 \| u + a_0 \|_{L^2(\Omega)} \right]
\| \nabla w \|_{L^2(\Omega)},
\quad
\langle \mathfrak{A}(w),w \rangle
\geq  \| \nabla w \|^2_{L^2(\Omega)},
\end{equation*}
where $C_{4 \hookrightarrow 2}$ denotes the best constant in $H_0^1(\Omega) \hookrightarrow L^4(\Omega)$. To derive the coercivity property we have used that $u \in \mathcal{A}_0$ so that $a_0 + u \geq 0$ a.e.~in $\Omega$. We now define
\begin{equation}
 \mathfrak{B}_k: H_0^1(\Omega) \rightarrow H^{-1}(\Omega),
 \qquad
 \langle \mathfrak{B}_k(w),v \rangle :=
 \langle \mathfrak{A}(w),v \rangle + \int_{\Omega}b_k(\cdot,w)v \mathrm{d}x.
\end{equation}
If we proceed as in the proof of \cite[Theorem 4.4]{MR2583281}, we can prove that $\mathfrak{B}_k$ is \emph{strongly monotone}, \emph{coercive}, and \emph{hemicontinuous} in $H_0^1(\Omega)$. Thus, an application of the main theorem on monotone operators \cite[Theorem 2.18]{MR3014456} yields the existence and uniqueness of $y_k \in H_0^1(\Omega)$, which solves \eqref{eq:weak_semi_bilinear_eq_b_k}. A stability estimate in $H_0^1(\Omega)$ follows from setting $v = y_k$ in \eqref{eq:weak_semi_bilinear_eq_b_k}. The bound $\| y_k \|_{L^{\infty}(\Omega)} \leq c_{\infty} \| f-a(\cdot,0) \|_{L^{q}(\Omega)}$ follows from the arguments developed in the proof of \cite[Theorem 4.5]{MR2583281}; see also \cite[Theorem B.2]{MR1786735} ($c_{\infty}$ is independent of $k$). Given this bound, we can deduce that $b_k(x,y_k(x)) = b(x,y_k(x))$ for $k$ sufficiently large and for a.e.~$x \in \Omega$. Consequently, $y_k$ solves the original problem \eqref{eq:weak_semi_bilinear_eq_b}. The uniqueness of solutions follows from \ref{A2} and the monotonicity of $b$.
\end{proof}

\subsection{Regularity results}
The following result shows that we can expect better Sobolev regularity properties when $f$ is slightly smoother.

\begin{theorem}[Sobolev regularity]\label{thm:cont_reg}
Let the assumptions of Theorem~\ref{thm:well-posedness-semilinearPDE} hold. If, in addition, $f\in L^{2}(\Omega)$, then there exists $\kappa > 4$ for $d = 2$ and $\kappa > 3$ for $d = 3$, so that the solution of \eqref{eq:weak_semi_bilinear_eq} belongs to $W^{1,\kappa}(\Omega) \cap C^{0,\varsigma}(\bar{\Omega})$. We also have the bound
\begin{equation}
\|\nabla y\|_{L^{\kappa}(\Omega)} 
+
\|y\|_{C^{0,\varsigma}(\bar \Omega)} 
\lesssim \|f - a(\cdot,0)\|_{L^{2}(\Omega)}\left( 1 + \| u \|_{L^2(\Omega)} \right),
\label{eq:regularity_bound_W1k}
\end{equation}
where $\varsigma$ is such that $0 < \varsigma \leq 1 - d/\kappa < 1$.
\label{thm:W1kappa-regularity}
\end{theorem}
\begin{proof}
We first note that \ref{A1}--\ref{A3} and $f \in L^{2}(\Omega)$ together with the fact that $y \in L^{\infty}(\Omega)$ (cf. Theorem \ref{thm:well-posedness-semilinearPDE}) allow us to prove that $g:= f - a(\cdot,y) - uy \in L^2(\Omega)$ and 
\begin{multline} 
\| g \|_{L^2(\Omega)} 
\leq 
\|  f - a(\cdot,0) \|_{L^2(\Omega)} 
+ 
\| a(\cdot,y) - a(\cdot,0) \|_{L^2(\Omega)} 
+
\|  uy \|_{L^2(\Omega)}
\\
\lesssim 
\|  f - a(\cdot,0) \|_{L^2(\Omega)} \left[ 1 + C_{a,\mathfrak{m}} + \|  u \|_{L^2(\Omega)} \right], 
\qquad 
\mathfrak{m} = \| y \|_{L^{\infty}(\Omega)}.
\label{eq:estimate_for_g}
\end{multline}
We have also used that $\|y\|_{L^{\infty}(\Omega)} \lesssim \| f-a(\cdot,0) \|_{L^{2}(\Omega)}$. We now perform a simple calculation based on the definition of negative Sobolev norms and standard Sobolev embeddings to conclude that $g \in W^{-1,\kappa}(\Omega)$ for some $\kappa > 4$ when $d=2$ and for some $\kappa > 3$ when $d=3$. The desired bound for $\|\nabla y\|_{L^{\kappa}(\Omega)}$ results from the application of \cite[Theorem 0.5]{MR1331981} ($\kappa$ may be further restricted if necessary). Finally, invoking \cite[Theorem 4.12, Part II]{adams2003sobolev}, we obtain that $y \in C^{0,\varsigma}(\bar{\Omega})$ along with the desired bound.
\end{proof}

\begin{remark}[H\"older regularity]
 \EO{Let the assumptions of Theorem~\ref{thm:well-posedness-semilinearPDE} hold. In the case where $f\in L^{q}(\Omega)$ and $a(\cdot,0) \in L^{q}(\Omega)$, for some $q>d/2$, the arguments used in the proof of Theorem~\ref{thm:cont_reg} allow us to derive the following regularity results: there exists $\mathtt{k} > 2$ for $d = 2$ and $\mathtt{k} > 3$ for $d = 3$ such that $y \in W^{1,\mathtt{k}}(\Omega) \cap H_0^1(\Omega)$. Since $\mathtt{k} > d$, an immediate application of \cite[Theorem 4.12, Part II]{adams2003sobolev} shows that $y \in C^{0,\nu}(\bar \Omega)$, where $\nu$ satisfies $0 < \nu \leq 1 - d/\mathtt{k} < 1$. In Theorem~\ref{thm:cont_reg}, however, we have assumed that $f \in L^2(\Omega)$. The fact that $a(\cdot,0) \in L^2(\Omega)$ follows from assumption \ref{A1}. These stronger regularity assumptions on the data are necessary to obtain the $H^2(\Omega)$-regularity result derived in Theorem~\ref{thm:H^2-reg-y} below, which is essential for analyzing finite element schemes \cite[Remark 2.4]{Casas_bilinear_2}. We emphasize that in this work, our intention is to lay the foundations for the development of finite element schemes.} 
\end{remark}

We continue with a standard $H^2(\Omega)$-regularity result on convex domains.

\begin{theorem}[$H^{2}(\Omega)$-regularity]
\label{thm:H^2-reg-y}
Let the assumptions of Theorem~\ref{thm:well-posedness-semilinearPDE} hold. If, in addition, $\Omega$ is convex and $f\in L^{2}(\Omega)$, then the solution $y$ of problem \eqref{eq:weak_semi_bilinear_eq} belongs to $H^{2}(\Omega)$ and $\|y\|_{H^{2}(\Omega)} \lesssim \|f - a(\cdot,0)\|_{L^{2}(\Omega)} (1 + \| u \|_{L^{2}(\Omega)})$.
\end{theorem}
\begin{proof}
The proof follows directly from \cite[Theorems 3.2.1.2 and 4.3.1.4]{MR3396210} for $d = 2$ and \cite[Theorem 3.2.1.2]{MR3396210} and \cite[section 4.3.1]{MR2641539} for $d = 3$ in conjunction with the fact that $g = f - a(\cdot,y) - uy \in L^2(\Omega)$ and satisfies the bound \eqref{eq:estimate_for_g}.
\end{proof}

\subsection{Differentiability properties}
We now examine differentiability properties of the map $u \mapsto y$, where $y$ corresponds to the solution of problem \eqref{eq:weak_semi_bilinear_eq}.

\begin{theorem}[differentiability properties of $u \mapsto y$]\label{thm:properties_S}
Let us assume that \ref{A1}--\ref{A3} hold and let $f \in L^{2}(\Omega)$. Let $D(\Omega) = \{y \in H_{0}^{1}(\Omega): \Delta y \in L^{2}(\Omega)\}$. Then there exists an open set $\mathcal{A}$ in $L^{2}(\Omega)$ such that $\mathcal{A}_{0} \subset \mathcal{A}$ and for every $u \in \mathcal{A}$ the problem \eqref{eq:weak_semi_bilinear_eq} has a unique solution $y \in H_{0}^{1}(\Omega) \cap C^{0,\varsigma}(\bar{\Omega})$, where $0 < \varsigma \leq 1-d/\kappa < 1$ and $\kappa$ is as in the statement of Theorem \ref{thm:W1kappa-regularity}. Moreover, there exists a map $\mathcal{S}: \mathcal{A} \rightarrow D(\Omega)$ of class $C^{2}$ so that for every $u \in \mathcal{A}$ we have the following properties:
\begin{itemize}
\item[(i)] $\mathcal{S}(u) = y \in H_{0}^{1}(\Omega) \cap C^{0,\varsigma}(\bar{\Omega})$, where $y$ is the unique solution to \eqref{eq:weak_semi_bilinear_eq},
\item[(ii)] for every $h \in L^{2}(\Omega)$, the function $z = \mathcal{S}'(u)h \in H_{0}^{1}(\Omega) \cap C^{0,\varsigma}(\bar{\Omega})$ is the unique solution to the problem
\begin{equation}\label{eq:DS}
(\nabla z, \nabla v)_{L^{2}(\Omega)} + \left(\tfrac{\partial a}{\partial y}(\cdot,y)z,v \right)_{L^{2}(\Omega)} + (uz,v)_{L^{2}(\Omega)} = -(hy,v)_{L^{2}(\Omega)}
\end{equation}
for all $v \in H^{1}_{0}(\Omega)$, and
\item[(iii)] for every $h_{1},h_{2} \in L^{2}(\Omega)$, the function $\gamma = \mathcal{S}''(u)h_{1}h_{2} \in H_{0}^{1}(\Omega) \cap C^{0,\varsigma}(\bar{\Omega})$ is the unique solution to the problem
\begin{multline}\label{eq:D2_S}
    (\nabla \gamma, \nabla v)_{L^{2}(\Omega)} + \left( \tfrac{\partial a}{\partial y}(\cdot, y)\gamma,v\right)_{L^{2}(\Omega)} + (u\gamma ,v)_{L^{2}(\Omega)} 
\\
= -(h_{1}z_{2} + h_{2}z_{1},v)_{L^{2}(\Omega)} - \left(\tfrac{\partial^{2}a}{\partial y^{2}}(\cdot,y)z_{2}z_{1},v\right)_{L^{2}(\Omega)}
\end{multline}
for all $v \in H^{1}_{0}(\Omega)$, where $z_i = \mathcal{S}'(u)h_i$ and $i \in \{1,2\}$.
\end{itemize}
\end{theorem}
\begin{proof}
We begin the proof by introducing the following norm in $D(\Omega)$:
\begin{equation}\label{eq:norm_D}
 \| y \|_{D(\Omega)} := \|\nabla y\|_{L ^{2}(\Omega)} + \| \Delta y \|_{L^{2}(\Omega)}, \qquad y \in D(\Omega).
\end{equation}
It can be proved that $D(\Omega)$, endowed with this norm, is a Banach space. We now prove that $D(\Omega) \hookrightarrow C^{0,\varsigma}(\bar{\Omega})$. In fact, let  $y \in D(\Omega)$. Note that $y$ can be seen as the weak solution to the following problem: Find $y \in H_0^1(\Omega)$ such that
$
-\Delta y = g
$
in $\Omega$ and $y = 0$ on $\partial \Omega$,
where $g \in L^{2}(\Omega)$. An application of \cite[Theorem 0.5]{MR1331981} shows that
\begin{equation*}
\| y \|_{C^{0,\varsigma}(\bar{\Omega})}
\lesssim
\| \nabla y \|_{L^{\kappa}(\Omega)}
\lesssim \| g \|_{L^{2}(\Omega)} \lesssim \| y \|_{D(\Omega)}.
\end{equation*} 

We now define
\[
F: L^{2}(\Omega) \times D(\Omega) \rightarrow L^{2}(\Omega), \qquad (u,y) \mapsto F(u,y) := -\Delta y + a(\cdot,y) + uy - f.
\]
Based on the assumptions \ref{A1}--\ref{A3}, it follows immediately that $F$ is well-defined and that $F$ is of class $C^{2}$. The rest of the proof follows from the arguments developed in \cite[Theorem 2.5]{Casas_bilinear_1} in combination with the arguments elaborated in the proofs of Theorems \ref{thm:well-posedness-semilinearPDE} and \ref{thm:W1kappa-regularity}. In our framework, the open set $\mathcal{A}$ is as follows:
\begin{equation}
\mathcal{A} \subset L^2(\Omega),
\qquad
\mathcal{A} := \bigcup_{\bar{u} \in \mathcal{A}_{0}} B_{\varepsilon_{\bar{u}}}(\bar{u}),
\qquad
0< \varepsilon_{\bar{u}} < C_{4\hookrightarrow2}^{-2}.
\label{eq:mathcal_A}
\end{equation}
Here, $B_{\varepsilon_{\bar{u}}}(\bar{u})$ denotes the open ball in $L^2(\Omega)$ centered at $\bar{u} \in \mathcal{A}_{0}$ of radius $\varepsilon_{\bar{u}}$ and $C_{4\hookrightarrow 2}$ corresponds to the best constant in the Sobolev embedding $H^{1}_{0}(\Omega) \hookrightarrow L^{4}(\Omega)$.
\end{proof}

As explained in the following remark, problems \eqref{eq:DS} and \eqref{eq:D2_S} are well-posed.

\begin{remark}[well-posedness of \eqref{eq:DS} and \eqref{eq:D2_S}]\label{rmk:wp_derivatives_problems} Let $u \in \mathcal{A}$ and $y = \mathcal{S}(u)$; $\mathcal{A}$ is defined in \eqref{eq:mathcal_A}. By construction, we have the existence of $\bar{u} \in \mathcal{A}_{0}$ and  $\varepsilon_{\bar{u}} < C_{4 \hookrightarrow 2}^{-2}$ such that $u \in B_{\varepsilon_{\bar{u}}}(\bar{u})$. Define the form $B : H_{0}^{1}(\Omega) \times H_{0}^{1}(\Omega) \rightarrow \mathbb{R}$ by
\[
(w,v) \mapsto B(w,v) := \int_{\Omega}\nabla w \cdot \nabla v \mathrm{d}x + \int_{\Omega}\left[\tfrac{\partial a}{\partial y}(\cdot,y) + \bar{u} \right]wv  \mathrm{d}x + \int_{\Omega}(u - \bar{u})wv \mathrm{d}x.
\]
It is clear that $B$ is a bilinear and continuous form in $H_{0}^{1}(\Omega) \times H_{0}^{1}(\Omega)$. Moreover, the bilinear form $B$ is coercive in $H_{0}^{1}(\Omega) \times H_{0}^{1}(\Omega)$. In fact, given $w \in H_0^1(\Omega)$, we have
\begin{equation}
B(w,w) \geq \|\nabla w\|^{2}_{L^{2}(\Omega)} - \|u - \bar{u}\|_{L^{2}(\Omega)}\|w\|^{2}_{L^{4}(\Omega)} \geq (1 - \varepsilon_{\bar{u}}C_{4 \hookrightarrow 2}^{2})\|\nabla w\|^{2}_{L^{2}(\Omega)},
\label{eq:B_coercive}
\end{equation}
where we have used $H_0^{1}(\Omega) \hookrightarrow L^{4}(\Omega)$, the assumption \ref{A2}, and the fact that $\bar{u} \in \mathcal{A}_0$. Given $h \in L^2(\Omega)$, it is clear that $L$ defined by $H_{0}^{1}(\Omega) \ni v \mapsto L(v) = -(hy,v)_{L^{2}(\Omega)} \in \mathbb{R}$ is linear and continuous. Therefore, a simple application of the Lax-Milgram lemma shows that \eqref{eq:DS} is well-posed. Moreover, using \eqref{eq:B_coercive} and \eqref{eq:stability_weak_problem} we can obtain
\[
\|\nabla z \|_{L^{2}(\Omega)} \leq C \|h\|_{L^{2}(\Omega)}\|\nabla y\|_{L^{2}(\Omega)} \leq C \|h\|_{L^{2}(\Omega)}\|f - a(\cdot, 0)\|_{H^{-1}(\Omega)},
\]
where $C = C^{2}_{4 \hookrightarrow 2} ( 1 - \varepsilon_{\bar{u}}C^{2}_{4 \hookrightarrow 2} )^{-1}$. The fact that $z \in C^{0,\varsigma}(\bar \Omega)$ follows from the arguments developed in the proof of Theorem \ref{thm:cont_reg}.
The same arguments show that \eqref{eq:D2_S} is also well-posed and similar arguments show that $\gamma \in C^{0,\varsigma}(\bar \Omega)$.
\end{remark}


\section{The optimal control problem}\label{sec:OCP}
In this section, we introduce a weak formulation of the optimal control problem \eqref{def:cost_func}--\eqref{def:box_const}, analyze the existence of optimal solutions, and study the well-posedness, as well as differentiability and regularity properties, of the associated adjoint problem. 
With these results at hand, we derive first and necessary and sufficient second order optimality conditions.

The weak formulation mentioned above reads as follows: Find
\begin{equation}\label{eq:weak_ocp}
    \min
    \left\lbrace J(y,u) : (y,u) \in H^{1}_{0}(\Omega) \cap C^{0,\varsigma}(\bar{\Omega})  
    \times \mathbb{U}_{ad}\right\rbrace
\end{equation}
subject to the \emph{monotone}, \emph{semilinear}, and \emph{elliptic} PDE
\begin{equation}\label{eq:weak_state_eq}
    (\nabla y, \nabla v)_{L^{2}(\Omega)} + (a(\cdot,y),v)_{L^{2}(\Omega)} + (uy,v)_{L^{2}(\Omega)} = (f,v)_{L^{2}(\Omega)}  \quad \forall v \in H^{1}_{0}(\Omega). 
\end{equation}
Here, $a : \Omega \times \mathbb{R} \rightarrow \mathbb{R}$ satisfies \ref{A1}--\ref{A3}, $f \in L^{2}(\Omega)$, and $\varsigma$ is such that $0 < \varsigma \leq 1 - d/\kappa < 1$, where $\kappa$ is as in the statement of Theorem \ref{thm:W1kappa-regularity}.

To perform an analysis for the optimal control problem \eqref{eq:weak_ocp}--\eqref{eq:weak_state_eq}, we make the following assumption \cite[Assumption 3.1]{Casas_bilinear_1}:
\begin{equation}\label{eq:restriction_a}
a_{0}(x)+\mathtt{a} \geq 0 
\quad
\text{for a.e. } x \in \Omega.
\end{equation}
From this assumption it follows directly that $\mathbb{U}_{ad} \subset \mathcal{A}_{0} \subset \mathcal{A}$. We recall that the set $\mathcal{A}_0$ is defined in \eqref{eq:mathcal_A_0} and the set $\mathcal{A}$ is given as in \eqref{eq:mathcal_A}.

Under assumption \eqref{eq:restriction_a}, we are in a position to apply the results of section \ref{sec:semilinear_eq} directly. In particular, we have for $u \in \mathbb{U}_{ad}$ the existence and uniqueness of a solution $y \in H^{1}_{0}(\Omega) \cap C^{0,\varsigma}(\bar{\Omega})$ to problem \eqref{eq:weak_state_eq}. We note that since $y \in C^{0,\varsigma}(\bar \Omega)$, the point evaluations of $y$ in \eqref{def:cost_func} are well-defined.

\subsection{Existence of solutions}

In Theorem \ref{thm:properties_S} we have proved the existence of the \emph{control-to-state map} $\mathcal{S}: \mathcal{A} \rightarrow H^{1}_{0}(\Omega) \cap C^{0,\varsigma}(\bar{\Omega})$, which assigns to a control $u$ the unique state $y$ that solves \eqref{eq:weak_state_eq}. With this map in hand, we introduce the notion of \emph{global solution}: $\bar{u} \in \mathbb{U}_{ad}$ is a global solution of \eqref{eq:weak_ocp}--\eqref{eq:weak_state_eq} if
$
J(\mathcal{S}(\bar{u}),\bar{u}) \leq J(\mathcal{S}(u),u)
$
for all $u \in \mathbb{U}_{ad}$. If $\bar{u}$ is a global solution, the pair $(\bar{y},\bar{u})$, where $\bar{y} = \mathcal{S} (\bar u)$, is called a \emph{globally optimal state--control pair}.

The existence of a globally optimal state--control pair is as follows.

\begin{theorem}[existence of global solutions]
The control problem \eqref{eq:weak_ocp}--\eqref{eq:weak_state_eq} admits at least one globally optimal state--control pair $(\bar{y},\bar{u}) \in H^{1}_{0}(\Omega) \cap C^{0,\varsigma}(\bar{\Omega}) \times \mathbb{U}_{ad}$.
\end{theorem}
\begin{proof}
The proof follows from an adaptation of the arguments elaborated in the proof of \cite[Theorem 4.15]{MR2583281} to our bilinear scenario with pointwise tracking. For brevity, we omit the details.
\end{proof}


\subsection{The adjoint equation}

To derive optimality conditions, we introduce the \emph{adjoint equation}. Let $u \in \mathcal{A}$ and let $y=\mathcal{S}(u)$ (cf. Theorem \ref{thm:properties_S}). Let $r$ be such that
\begin{equation}
\label{eq:r}
 r \in (1,2) \textrm{ if } d=2,
 \qquad
 r \in \left[ \tfrac{6}{5},\tfrac{3}{2} \right) \textrm{ if } d=3,
\end{equation}
and let $s$ be the H\"older conjugate of $r$, i.e., $s$ is such that $1/r + 1/s = 1$. We note that $s>d$. In this framework, we introduce the \emph{adjoint equation} as follows: Find $p \in W^{1,r}_{0}(\Omega)$ such that
\begin{equation}\label{eq:weak_adj_eq}
 \int_{\Omega}\nabla w \cdot \nabla p \mathrm{d}x +  \int_{\Omega}\left[ \tfrac{\partial a}{\partial y}(\cdot,y) + u \right]pw \mathrm{d}x = \sum_{t \in \mathcal{D}} \langle (y(t)-y_{t})\delta_{t},w \rangle \quad \forall w \in W^{1,s}_{0}(\Omega).
\end{equation} 
Here, $\langle \cdot , \cdot \rangle$ denotes the duality pairing between $W^{-1,r}(\Omega)$ and $W^{1,s}_{0}(\Omega)$. We note that the restriction $r \geq 6/5$ in three dimensions in \eqref{eq:r} guarantees that the term $([ \partial a/ \partial y (\cdot,y) + u ]w,p )_{L^{2}(\Omega)}$ in \eqref{eq:weak_adj_eq} is well-defined because
\[
p \in W^{1,r}_{0}(\Omega) \hookrightarrow L^2(\Omega),
\qquad
w \in W^{1,s}_{0}(\Omega) \hookrightarrow C(\bar \Omega);
\]
see \cite[Theorem 4.12, Part I, Case C]{adams2003sobolev} and \cite[Theorem 4.12, Part II]{adams2003sobolev}, respectively.

\subsubsection{Well-posedness}
We now establish the well-posedness of the adjoint problem \eqref{eq:weak_adj_eq}. To do so, we consider two cases: $u \in \mathcal{A}_0$ and $u \in \mathcal{A} \setminus \mathcal{A}_0$.

\begin{theorem}[well-posedness of the adjoint problem I]\label{thm:adjoint_problem}
Let $u \in \mathcal{A}_0$, let $y=\mathcal{S}(u)$, and let $r$ be as in \eqref{eq:r}. Then, there is a unique solution $p \in W^{1,r}_{0}(\Omega)$ for \eqref{eq:weak_adj_eq}. Moreover, we have the bound
\begin{equation}\label{eq:stab_adj_eq}
\|\nabla p\|_{L^{r}(\Omega)} \lesssim \|f - a(\cdot,0)\|_{L^{2}(\Omega)} + \sum_{t \in \mathcal{D}}|y_{t}|,
\end{equation}
where the hidden constant depends on $\# \mathcal{D}$ but is independent of $p$.
\end{theorem}

\begin{proof}
If $u \in \mathcal{A}_{0}$, the proof follows from slight modifications of the arguments developed in the proof of \cite[Th\'eor\`eme 9.1]{MR0192177}. It is important to note that in our framework the strict positivity of the reactive coefficient required in \cite[Th\'eor\`eme 9.1, assumptions (9.2) and (9.2')]{MR0192177} can be relaxed so that it is nonnegative a.e.~in $\Omega$. 
\end{proof}

To analyze the well-posedness of \eqref{eq:weak_adj_eq} in the case that $u \in \mathcal{A}\setminus \mathcal{A}_0$, we introduce the following bilinear form:
$Q_{u}: W^{1,r}_{0}(\Omega) \times W^{1,s}_{0}(\Omega) \rightarrow \mathbb{R}$, where
\[
Q_{u}(v,w) :=  \int_{\Omega} \nabla w \cdot \nabla v \mathrm{d}x 
 + 
\int_{\Omega}
\left[\frac{\partial a}{\partial y}(\cdot,y) +u\right]vw 
\mathrm{d}x,
\qquad
u \in L^2(\Omega).
\]
It is immediate that $Q_u$ is well-defined and continuous in $W^{1,r}_{0}(\Omega) \times W^{1,s}_{0}(\Omega)$. Moreover, the well-posedness result of Lemma \ref{thm:adjoint_problem} directly implies the following inf-sup condition for $u \in \mathcal{A}_0$ \cite[Theorem 2.6]{MR2050138}: there exists $\beta >0$ such that
 \begin{equation}\label{eq:inf_sup_u_in_A0}
\beta \|\nabla v\|_{L^{r}(\Omega)} \leq \sup_{w \in W^{1,s}_{0}(\Omega)}\dfrac{Q_{u}(v,w)}{\|\nabla w\|_{L^{s}(\Omega)}} 
\quad \forall v \in W_{0}^{1,r}(\Omega).
\end{equation}

To present the next result, we further restrict $\varepsilon_{\bar{u}}$ in the definition of $\mathcal{A}$ in \eqref{eq:mathcal_A}:
\begin{equation}
\mathcal{A} \subset L^2(\Omega),
\qquad
\mathcal{A} = \bigcup_{\bar{u} \in \mathcal{A}_{0}} B_{\varepsilon_{\bar{u}}}(\bar{u}),
\qquad
0< \varepsilon_{\bar{u}} < \min \{C_{4\hookrightarrow2}^{-2}, \beta C_{2\hookrightarrow r}^{-1}C_{\infty \hookrightarrow s}^{-1}\},
\label{eq:mod_mathcal_A}
\end{equation}
where $C_{2\hookrightarrow r}$ and $C_{\infty\hookrightarrow s}$ denote the constants involved in $W^{1,r}_{0}(\Omega) \hookrightarrow L^{2}(\Omega)$ and $W^{1,s}_{0}(\Omega) \hookrightarrow C(\bar{\Omega})$, respectively. The well-posedness of the problem \eqref{eq:weak_adj_eq} for $u \in \mathcal{A} \setminus \mathcal{A}_0$ is therefore as follows.

\begin{theorem}[well-posedness of the adjoint problem II]\label{thm:well_posed_adjoint2}
Let $u \in \mathcal{A} \setminus \mathcal{A}_0$ and let $y=\mathcal{S}(u)$. Let $r$ be as in \eqref{eq:r}. Then, there is a unique solution $p \in W^{1,r}_{0}(\Omega)$ for \eqref{eq:weak_adj_eq} that satisfies the bound \eqref{eq:stab_adj_eq} with a hidden constant that depends on $u$ and $\# \mathcal{D}$ but is independent of $p$.
\end{theorem}
\begin{proof}
Let $u \in \mathcal{A} \setminus \mathcal{A}_0$. We rewrite the adjoint problem \eqref{eq:weak_adj_eq} as follows:
\begin{equation}\label{eq:adj_problem_op}
p \in W^{1,r}_{0}(\Omega): \quad Q_{u}(p,w) = \sum_{t \in \mathcal{D}}\left<(y(t)-y_{t})\delta_{t} ,w \right> \quad \forall w \in W^{1,s}_{0}(\Omega).
\end{equation}
To examine the well-posedness of \eqref{eq:adj_problem_op}, we verify the conditions (BNB1) and (BNB2) in \cite[Theorem 2.6]{MR2050138} for $Q_{u}$. 

(BNB1): We first verify the inf-sup condition for $Q_{u}$ in \cite[Theorem 2.6, (BNB1)]{MR2050138}. Since $u \in \mathcal{A}$, from \eqref{eq:mod_mathcal_A} there exists $\bar{u} \in \mathcal{A}_{0}$ and an open ball $B_{\varepsilon_{\bar{u}}}(\bar{u}) \subset L^{2}(\Omega)$ such that $u \in B_{\varepsilon_{\bar{u}}}(\bar{u})$. Let us now observe that for all $v \in W^{1,r}_{0}(\Omega)$ and for all $w \in W^{1,s}_{0}(\Omega)$, $Q_{u}(v,w) = Q_{\bar{u}}(v,w) + ((u-\bar{u}),vw)_{L^{2}(\Omega)}$. It therefore follows from \eqref{eq:inf_sup_u_in_A0} that
\begin{equation*}
\tilde{\beta} \|\nabla v\|_{L^{r}(\Omega)} \leq \sup_{w \in W^{1,s}_{0}(\Omega)}\dfrac{Q_{u}(v,w)}{\|\nabla w\|_{L^{s}(\Omega)}} \qquad \forall v \in W_{0}^{1,r}(\Omega),
\end{equation*}
where $\tilde{\beta} := \beta - C_{2\hookrightarrow r}C_{\infty\hookrightarrow s}\varepsilon_{\bar{u}}$. The modification in \eqref{eq:mod_mathcal_A} leads to the conclusion that $\tilde{\beta}>0$.

(BNB2): We now verify the second condition for $Q_{u}$ in \cite[Theorem 2.6, (BNB2)]{MR2050138}. Let $w \in W_{0}^{1,s}(\Omega)$ be such that
\begin{equation}\label{eq:bilinear_adj_eq_zero}
Q_{u}(v,w) = 0 \quad \forall v \in W^{1,r}_{0}(\Omega). 
\end{equation}
We need to prove that $w \equiv 0$. To do this, we first note that \eqref{eq:bilinear_adj_eq_zero} holds for all $v \in W_{0}^{1,s}(\Omega)$. In fact, since $r<s$, it follows that $W_{0}^{1,s}(\Omega) \hookrightarrow W_{0}^{1,r}(\Omega)$. We can therefore set $v=w$ in \eqref{eq:bilinear_adj_eq_zero} and proceed as in Remark \ref{rmk:wp_derivatives_problems} to obtain that
\begin{equation*}
0 = Q_{u}(w,w) \geq  (1-\varepsilon_{\bar{u}}C^{2}_{4 \hookrightarrow 2})\|\nabla w\|^{2}_{L^{2}(\Omega)} \geq 0.
\end{equation*} 
From this inequality and \eqref{eq:mod_mathcal_A} it follows that $w \equiv 0$, as we intended to show. 

The stability estimate \eqref{eq:stab_adj_eq} follows from an application of the bound (2.5) in \cite[Theorem 2.6]{MR2050138}.
\end{proof}

\subsubsection{\EO{The control-to-adjoint state map: differentiability}}

\EO{We define the control-to-adjoint state map as follows: $\Phi: \mathcal{A} \rightarrow W_{0}^{1,r}(\Omega)$, which, given a control $u \in \mathcal{A}$, associates to it the unique adjoint state $p = \Phi(u)$ that solves \eqref{eq:weak_adj_eq}. Here, $\mathcal{A}$ is described in \eqref{eq:mod_mathcal_A}. In light of Theorems \ref{thm:adjoint_problem} and \ref{thm:well_posed_adjoint2}, the map $\Phi$ is well-defined. In the following result, we show that $\Phi$ is differentiable.}


\begin{theorem}[differentiability properties of $u \mapsto p$]\label{thm:diff_u_p}
\EO{Let $r$ be as in \eqref{eq:r}. Then, the map $\Phi: \mathcal{A} \rightarrow W_{0}^{1,r}(\Omega)$ is of class $C^{1}$. Moreover, for every $u \in \mathcal{A}$ and for every $h \in L^{2}(\Omega)$, the function $\eta = \Phi'(u)h \in W_{0}^{1,r}(\Omega)$ is the unique solution to}
\begin{multline}\label{eq:DPhi}
 \int_{\Omega}\nabla w \cdot \nabla \eta \mathrm{d}x +  \int_{\Omega}\left[ \frac{\partial a}{\partial y}(\cdot,y) + u \right]\eta w \mathrm{d}x \\
= \sum_{t \in \mathcal{D}} \langle z(t)\delta_{t},w \rangle - \int_{\Omega}\left[ \frac{\partial^{2} a}{\partial y^{2}}(\cdot,y)z + h \right]pw \mathrm{d}x
\quad \forall w \in W_0^{1,s}(\Omega),
\end{multline}  
where $r^{-1} + s^{-1} = 1$, $y = \mathcal{S}(u)$, and $z = \mathcal{S}'(u)h$.

\end{theorem}
\begin{proof} 
First, we show that problem \eqref{eq:DPhi} is well-posed. To this end, we note that the involved bilinear form $Q_u$ satisfies the conditions (BNB1) and (BNB2) in \cite[Theorem 2.6]{MR2050138} provided that $u \in \mathcal{A}$; see Theorems \ref{thm:adjoint_problem} and \ref{thm:well_posed_adjoint2}. We also note that, since $y \in H_0^1(\Omega) \cap C^{0,\varsigma}(\bar \Omega)$ (cf.~Theorem \ref{thm:cont_reg}), $z \in H_0^1(\Omega) \cap C^{0,\varsigma}(\bar \Omega)$ (cf.~Theorem \ref{thm:properties_S}) and $p \in W_0^{1,r}(\Omega)$ (cf.~Theorems \ref{thm:adjoint_problem} and \ref{thm:well_posed_adjoint2}), it can be proved that  the forcing term in \eqref{eq:DPhi} belongs to $W^{-1,r}(\Omega)$.  Thus, we have all the ingredients for the application of \cite[Theorem 2.6]{MR2050138} and can deduce that problem \eqref{eq:DPhi} is well-posed.

To prove \EO{the differentiability of $\Phi$}, we define
\[
G: \mathcal{A} \times W_0^{1,r}(\Omega) \rightarrow W^{-1,r}(\Omega), 
\quad 
G(u,p) := -\Delta p + \left( \tfrac{\partial a}{\partial y}(\cdot,y)+u \right)p  - \sum_{t \in \mathcal{D}}(y(t)-y_{t})\delta_{t},
\]
where $y = \mathcal{S}(u)$. From \ref{A1}--\ref{A3} it follows that $G$ is well-defined and that $G$ is of class $C^1$. Moreover, for $(\bar{u},\bar{p}) \in \mathcal{A} \times W_0^{1,r}(\Omega)$, the derivative $\tfrac{\partial G}{\partial p} (\bar u, \bar p)$ is given by
\[
\frac{\partial G}{\partial p} (\bar u, \bar p) : W_0^{1,r}(\Omega) \rightarrow W^{-1,r}(\Omega),
\quad
h \mapsto -\Delta h + \left( \frac{\partial a}{\partial y} (\cdot,\bar{y}) + \bar{u} \right) h,
\]
where $\bar{y} = \mathcal{S}(\bar{u})$. The map $\partial G/ \partial p (\bar u, \bar p)$ is linear and continuous. It satisfies the bound

\begin{equation}
\left\| \frac{\partial G} {\partial p} (\bar u, \bar p) h \right \|_{W^{-1,r}(\Omega)} \lesssim (1 + C_{a,\mathfrak{m}} |\Omega|^{\frac{1}{2}} + \| \bar u \|_{L^2(\Omega)}) \| \nabla h \|_{L^r(\Omega)}
\quad
\forall h \in W_0^{1,r}(\Omega).
\label{eq:dfdu_bound}
\end{equation}
Applying the results of Theorems \ref{thm:adjoint_problem} and \ref{thm:well_posed_adjoint2}, we deduce that $\partial G/ \partial p (\bar u, \bar p)$ is a bijection. From this, from the bound \eqref{eq:dfdu_bound}, and from the application of the open mapping theorem, we can conclude that $\partial G/ \partial p (\bar u, \bar p)$ is an isomorphism. \EO{An application of the implicit function theorem thus implies that $\Phi$ is of class $C^{1}$.}
Finally, the fact that $\eta = \Phi'(u)h$ is the solution of \eqref{eq:DPhi} follows from differentiating the relation $G(u,\Phi(u)) = 0$. This concludes the proof.
\end{proof}

\subsubsection{Regularity results}
We now analyze suitable regularity properties for the solution of the adjoint problem, which are important to derive regularity properties for locally optimal controls. To begin our analysis, we introduce the set 
\begin{equation}\label{eq:set_mathcal_E}
\mathcal{E}:=\left\lbrace t \in \mathcal{D}: y(t)\neq y_{t} \right\rbrace.
\end{equation}
We immediately note that if $\mathcal{E} = \emptyset$, then $p \equiv 0$.

Following the arguments in the proofs of \cite[Theorem 3.4]{MR2434067} and \cite[Lemma 5.2]{MR3973329}, we now state and prove the following regularity result.

\begin{theorem}[local regularity]\label{thm:H^2-reg-p} 
Let $r$ be as in \eqref{eq:r}. Let $p \in W^{1,r}_{0}(\Omega)$ be the solution of problem \eqref{eq:weak_adj_eq}, where $y = \mathcal{S}(u)$ and $u \in \mathbb{U}_{ad}$. If $\Omega$ is a convex polytope, then
\begin{equation}\label{eq:H2_Holder_reg_adj}
p \in H^{2}\left(\Omega \setminus \cup_{t \in \mathcal{E}} \bar{B}_{t}\right) \cap C^{0,1}(\bar{\Omega} \setminus \cup_{t \in \mathcal{E}}B_{t}).
\end{equation}
Here, $B_{t} \subset \Omega$ denotes an open ball centered at $t \in \mathcal{E}$ of strictly positive radius.
\end{theorem} 
\begin{proof}
We assume that $\mathcal{E} \neq \emptyset$ and proceed as in the proof of \cite[Lemma 5.2]{MR3973329} and \cite[Theorem 3.4]{MR4438718}. For every $t \in \mathcal{E}$, let $B_{t}$ and $C_{t}$ be open balls with center in $t$ and strictly positive radii such that $B_{t} \Subset C_{t} \subset \Omega$. Let $\psi: \Omega \rightarrow [0,1]$ be a smooth function satisfying the following conditions:
\begin{equation}
 \psi \equiv 1 \textrm{  in  } \Omega \setminus \bigcup_{t \in \mathcal{E}}\bar{C}_{t},
 \qquad
 \psi \equiv 0  \textrm{  in  } \bigcup_{t \in \mathcal{E}}B_{t},
 \qquad
 \psi > 0 \textrm{  in  } \bigcup_{t \in \mathcal{E}}C_{t} \setminus \bigcup_{t \in \mathcal{E}}\bar{B}_{t}.
 \label{eq:properties_of_psi}
\end{equation}
Define $\Psi := \psi p$. Since $p$ corresponds to the weak solution of \eqref{eq:weak_adj_eq}, simple calculations show that $\Psi$ can be seen as the distributional solution of the following problem:
\begin{equation}\label{eq:psi_equation}
-\Delta \Psi = -p\Delta \psi-2\nabla p \cdot \nabla \psi - \tfrac{\partial a}{\partial y}(\cdot,y)p\psi -up\psi
\text{ in } \Omega, \quad \Psi = 0
\text{ on } \partial \Omega.
\end{equation} 
Define $\mathfrak{g}:= -p\Delta \psi-2\nabla p \cdot \nabla \psi - \tfrac{\partial a}{\partial y}(\cdot,y)p\psi -up\psi$. We note that
\[
\textrm{supp}\left[ \sum_{t \in \mathcal{D}}  (y(t) - y_t) \delta_t \right] = \bigcup_{t \in \mathcal{E}} \{ t \}
\implies
\left[ \sum_{t \in \mathcal{D}}  (y(t) - y_t) \delta_t \right] \psi \equiv 0,
\]
because of the construction of $\psi$ (cf. \eqref{eq:properties_of_psi}). From the properties of $p$, $\psi$,  $a$, $y$, and $u$, it follows that $\mathfrak{g} \in H^{-1}(\Omega)$. This implies that there is a unique $\Psi \in H_0^1(\Omega)$ which solves \eqref{eq:psi_equation}. Moreover, it can also be proved that $\mathfrak{g} \in L^{r}(\Omega)$.
This, the convexity of $\Omega$, and the regularity result of \cite[Theorem 4.3.2.4]{MR3396210} for $d = 2$ and \cite[Theorem 4.3.2]{MR2641539} for $d = 3$
guarantee that $\Psi \in H_0^1(\Omega) \cap W^{2,r}(\Omega)$. This regularity result and the properties of $\psi$ show that $p \in W^{2,r}(\Omega \setminus \cup_{t \in \mathcal{E}}\bar{B}_{t})$. 

We now use a bootstrap argument to improve the previously derived regularity result for $p$. For this purpose, we use the embedding $W^{2,r}(\Omega \setminus \cup_{t \in \mathcal{E}}\bar{B}_{t}) \hookrightarrow H^{1}(\Omega \setminus \cup_{t \in \mathcal{E}}\bar{B}_{t})$, which holds because $r$ is as in \eqref{eq:r}, to obtain that $\mathfrak{g} \in L^{2}(\Omega)$. By applying \cite[Theorem 4.3.1.4]{MR3396210} for $d = 2$ and \cite[Theorem 4.3.2]{MR2641539} for $d = 3$, we obtain that $\Psi \in H ^{2}(\Omega)$. From this we conclude that $p \in H^{2}(\Omega \setminus \cup_{t \in \mathcal{E}}\bar{B}_{t})$. Based on a Sobolev embedding, this shows that $p \in W^{1,6}(\Omega \setminus \cup_{t \in \mathcal{E}}\bar{B}_{t})$. A second application of a bootstrap argument with the improved derived regularity of $p$, in conjunction with \cite[Lemma 4.1]{MR3973329}, allow us to obtain that $\Psi \in W^{1,\infty}(\Omega)$. Relying on \cite[Theorem 4.1]{MR2177410}, which provides the equivalence $W^{1,\infty}(\Omega)=C^{0,1}(\bar{\Omega})$, and the regularity properties of $\psi$, we can finally conclude that $p \in C^{0,1}(\bar{\Omega} \setminus \cup_{t \in \mathcal{E}}B_{t})$.
\end{proof} 


\subsection{Optimality conditions}

In this section, we analyze first and second order optimality conditions for problem \eqref{eq:weak_ocp}--\eqref{eq:weak_state_eq}. Since this optimal control problem is not convex, the optimality conditions are discussed in the context of local solutions in $L^2(\Omega)$ \cite[page 207]{MR2583281}. We say that $\bar{u} \in \mathbb{U}_{ad}$ is a \emph{local solution} of problem \eqref{eq:weak_ocp}--\eqref{eq:weak_state_eq} in the sense of $L^{2}(\Omega)$ if there exists $\varepsilon >0$ such that
\begin{equation*}
J(\mathcal{S}(\bar{u}),\bar{u}) \leq J(\mathcal{S}(u),u) \quad \forall u \in \mathbb{U}_{ad}: \| u - \bar{u}\|_{L^{2}(\Omega)} \leq \varepsilon.
\end{equation*}
$\bar{u} \in \mathbb{U}_{ad}$ is called a \emph{strict local solution} in the sense of $L^{2}(\Omega)$ if there exists $\varepsilon>0$ such that
$
J(\mathcal{S}(\bar{u}),\bar{u}) < J(\mathcal{S}(u),u)
$
for all  $u \in \mathbb{U}_{ad} \setminus \{ \bar u\}$ such that $\| u - \bar{u}\|_{L^{2}(\Omega)} \leq \varepsilon$.

\subsubsection{The reduced cost functional}
To derive first and second order optimality conditions, we introduce the reduced cost functional $j$ as follows:
\begin{equation}
j: \mathcal{A} \rightarrow \mathbb{R}, \qquad j(u) := J(\mathcal{S}(u),u).
\end{equation}
We recall that $\mathcal{A} \subset L^2(\Omega)$ is defined in \eqref{eq:mod_mathcal_A}. We now analyze differentiability properties of the cost functional $j$.

\begin{lemma}[differentiability of $j$]
\label{lem:diff_j}
Let $u \in \mathcal{A}$, let $y = \mathcal{S}(u)$, and let $p = \Phi (u)$, where $\Phi$ is defined in Theorem \ref{thm:diff_u_p}.
The functional $j$ is of class $C^{2}$ and the derivatives of $j$ are given by the following expressions:
\begin{align}
j'(u)h & = \int_{\Omega}(\alpha u - yp)h \mathrm{d}x \quad \forall h \in L^{2}(\Omega), \label{eq:Dj}
\\
j''(u)(h_{1},h_{2}) & = \int_{\Omega}(\alpha h_{2} - z_{2}p - y\eta_{2})h_{1} \mathrm{d}x = \int_{\Omega}(\alpha h_{1} - z_{1}p - y\eta_{1})h_{2} \mathrm{d}x \label{eq:DDj}
\end{align}
for all $h_{1}, h_{2} \in L^{2}(\Omega)$. Here, $z_{i} = \mathcal{S}'(u)h_{i}$ and $\eta_{i} = \Phi'(u)h_{i}$, where $i \in \{1,2\}$.
\end{lemma}
\begin{proof}
The fact that $j$ belongs to the class $C^{2}$ follows from the chain rule in conjunction with the fact that $\mathcal{S}$ belongs to the class $C^2$; see Theorem \ref{thm:properties_S}.

Let us now derive \eqref{eq:Dj}. Let $h \in L^{2}(\Omega)$ and note that
\begin{equation}
j'(u)h = \sum_{t \in \mathcal{D}}(y(t) - y_{t}) \cdot z(t) + (\alpha u,h)_{L^{2}(\Omega)},
\label{eq:j_prime}
\end{equation}
where $z = \mathcal{S}'(u)h$ solves \eqref{eq:DS}. Thus, it suffices to analyze $\sum_{t \in \mathcal{D}}(y(t)-y_{t}) \cdot z(t)$.
To do this, we set $w = z$ in \eqref{eq:weak_adj_eq}. This is possible because $z$ belongs to $H_0^1(\Omega) \cap W^{1,\kappa}(\Omega)$ for some $\kappa > 4$ in two dimensions and for some $\kappa > 3$ in three dimensions; see Theorem \ref{thm:properties_S}. From this it can be deduced that
\begin{equation}\label{eq:aux_1}
\int_{\Omega}  
\left [
\nabla z \cdot \nabla p 
+ 
\left( 
\frac{\partial a}{\partial y}(\cdot,y) + u \right) p z
\right]
\mathrm{d}x
= 
\sum_{t \in \mathcal{D}}(y(t)-y_{t}) \cdot z(t).
\end{equation}
We would now like to set $v=p$ in the problem that solves $z$, i.e., problem \eqref{eq:DS}. Unfortunately, this is not possible because $p$ does not belong to $H_0^1(\Omega)$. To solve this situation, we use a density argument as in the proof of \cite[Theorem 1]{MR3679932}: Let $\lbrace p_{k} \rbrace_{k \in \mathbb{N}} \subset C^{\infty}_{0}(\Omega)$  be such that $p_{k} \rightarrow p$ in $W^{1,r}_{0}(\Omega)$ as $k \uparrow \infty$. Since $p_k \in H_0^1(\Omega)$, we can set $v = p_{k}$ in \eqref{eq:DS} and obtain the following relation
\begin{equation}\label{eq:aux_k}
\int_{\Omega}  
\left [ 
\nabla z \cdot \nabla p_{k}
+ \left(\frac{\partial a}{\partial y}(\cdot,y) + u \right)z p_{k}
\right] \mathrm{d}x
= - \int_{\Omega} hy p_{k} \mathrm{d}x
\qquad 
\forall k \in \mathbb{N}.
\end{equation}
From the fact that $p \in W_0^{1,r}(\Omega)$ for every $r<d/(d-1)$ and in particular for an $r$ arbitrarily close to $d/(d-1)$, we deduce that $z \in H_{0}^{1}(\Omega) \cap W^{1,\kappa}(\Omega) \hookrightarrow W^{1,s}_{0}(\Omega)$. We therefore use that $y \in H_{0}^{1}(\Omega) \cap C^{0,\varsigma}(\bar{\Omega})$ (cf.~Theorem \ref{thm:cont_reg}), that $u,h \in L^{2}(\Omega)$, and that $p_{k} \rightarrow p$ in $W^{1,r}_{0}(\Omega)$ as $k \uparrow \infty$, to obtain
\begin{align*}
\left| \int_{\Omega} \nabla z \cdot \nabla p \mathrm{d}x -
 \int_{\Omega} \nabla z \cdot \nabla p_k
 \mathrm{d}x \right|
 \leq  \| \nabla z \|_{L^{s}(\Omega)} \|\nabla(p-p_{k}) \|_{L^{r}(\Omega)} \rightarrow 0,
 \\
 \left| \int_{\Omega} \frac{\partial a}{\partial y}(\cdot,y)zp \mathrm{d}x -  \int_{\Omega} \frac{\partial a}{\partial y}(\cdot,y)z p_{k}\mathrm{d}x \right| \leq C_{a,\mathfrak{m}} \|z\|_{C(\bar \Omega)}\|p-p_{k}\|_{L^{1}(\Omega)} \rightarrow 0,
 \\
\left| \int_{\Omega} hyp \mathrm{d}x  - \int_{\Omega} hyp_{k} \mathrm{d}x \right|
 \leq \|h\|_{L^{2}(\Omega)} \|y\|_{C(\bar \Omega)} \|p - p_{k}\|_{L^{2}(\Omega)} \rightarrow 0,
\end{align*}
and $ (uz,p_{k})_{L^{2}(\Omega)}\rightarrow (uz,p)_{L^{2}(\Omega)}$ as $k \uparrow \infty$. Taking the limit in \eqref{eq:aux_k} as $k \uparrow \infty$, we can therefore conclude that
\begin{equation}\label{eq:aux_2}
\int_{\Omega} \left[ \nabla z \cdot \nabla p \mathrm{d}x
+
\left( \frac{\partial a}{\partial y}(\cdot,y) + u \right) z p \right] \mathrm{d}x
= -(hy,p)_{L^{2}(\Omega)}.
\end{equation}
Comparing \eqref{eq:aux_2} and \eqref{eq:aux_1}, we obtain $-(hy,p)_{L^{2}(\Omega)} = \sum_{t \in \mathcal{D}}(y(t)-y_{t}) \cdot z(t)$. Replacing this identity into \eqref{eq:j_prime}, we obtain the desired identity \eqref{eq:Dj}.

Finally, \eqref{eq:DDj} follows from \eqref{eq:Dj}, Theorems \ref{thm:properties_S} and \ref{thm:diff_u_p}, and the fact that $j$ is of class $C^{2}$, which guarantees that $j''(u)(h_1,h_2) = j''(u)(h_2,h_1)$ for every $h_1, h_2 \in L^2(\Omega)$; see \cite[Remark 1.2.2]{MR4487366}. For the sake of brevity, we omit the details.
\end{proof}

\subsubsection{First order optimality conditions}

From the differentiability properties obtained in Lemma \ref{lem:diff_j}, we can derive first order optimality conditions.

\begin{theorem}[necessary first order optimality conditions]
Every locally optimal control $\bar{u} \in \mathbb{U}_{ad}$ satisfies the following variational inequality:
\begin{equation}\label{eq:char_var_ineq}
(\alpha \bar{u}-\bar{y}\bar{p},u-\bar{u})_{L^{2}(\Omega)} \geq 0 \quad \forall u \in \mathbb{U}_{ad},
\end{equation}
where $\bar{y} = \mathcal{S}(\bar{u})$ and $\bar{p} = \Phi (\bar{u})$, i.e., $\bar{p}$ is the solution to \eqref{eq:weak_adj_eq} with $u$ and $y$ replaced by $\bar{u}$ and $\bar{y}$, respectively.
\end{theorem}
\begin{proof}
If $\bar u \in \mathbb{U}_{ad}$ is a locally optimal control for problem \eqref{eq:weak_ocp}--\eqref{eq:weak_state_eq}, then $j'(\bar u) (u - \bar{u}) \geq 0$ for all $u \in \mathbb{U}_{ad}$ \cite[Lemma 4.18]{MR2583281}. Thus, the desired result follows directly from the characterization of $j'$ given in \eqref{eq:Dj}.
\end{proof}

\subsubsection{Second order optimality conditions}

We now establish necessary and sufficient second order optimality conditions. To do so, we first introduce some preliminary considerations. Let $(\bar{y},\bar{p},\bar{u}) \in H_{0}^{1}(\Omega) \cap C^{0,\varsigma}(\bar{\Omega}) \times W_{0}^{1,r}(\Omega) \times \mathbb{U}_{ad}$ satisfy the first order optimality conditions \eqref{eq:weak_semi_bilinear_eq}, \eqref{eq:weak_adj_eq}, and  \eqref{eq:char_var_ineq}. Define $\bar{\mathfrak{p}}:= \alpha \bar{u}-\bar{y}\bar{p}$. \EO{Since $\bar{u} \in \mathbb{U}_{ad} \subset \mathcal{A} \subset L^{2}(\Omega)$ and $\bar{p} \in W_{0}^{1,r}(\Omega) \hookrightarrow L^{2}(\Omega)$ (recall that $r$ is as in \eqref{eq:r}), it follows that $\bar{\mathfrak{p}} \in L^{2}(\Omega)$. On the other hand, we}
note that the variational inequality \eqref{eq:char_var_ineq} implies for a.e.~$x \in \Omega$:
\begin{equation}\label{eq:sign_con_j'}
     \bar{\mathfrak{p}}(x) \geq 0 \text{ if } \bar{u}(x) = \mathtt{a}, \qquad \bar{\mathfrak{p}}(x) \leq 0  \text{ if } \bar{u}(x) = \mathtt{b}, \qquad \bar{\mathfrak{p}}(x) = 0 \text{ if } \mathtt{a} < \bar{u}(x) < \mathtt{b}.
\end{equation}

We now introduce the \textit{cone of critical directions}
\begin{equation}\label{eq:cone_critical}
    C_{\bar{u}} := \left\lbrace 
 h \in L^{2}(\Omega) \text{ satisfying  \eqref{eq:sign_cond} and } h(x) = 0 \text{ if } \bar{\mathfrak{p}}(x) \neq 0 \ \mathrm{a.e.}~x \in \Omega \right\rbrace, 
\end{equation}
 where the condition \eqref{eq:sign_cond} reads as follows:
\begin{equation}\label{eq:sign_cond}
    h(x) \geq 0 \text{ a.e } x \in \Omega \text{ if } \bar{u}(x) = \mathtt{a}, \qquad h(x) \leq 0 \text{ a.e } x \in \Omega \text{ if } \bar{u}(x) = \mathtt{b}; 
\end{equation}
compare with \cite[definition (3.8)]{Casas_bilinear_1} and \cite[definition (2.14)]{Casas_bilinear_2}.

With these definitions, we can now formulate and prove necessary and sufficient second order optimality conditions.

\begin{theorem}[second order optimality conditions]\label{thm:second_order_cond}
    If $\bar{u} \in \mathbb{U}_{ad}$ is a locally optimal control for \eqref{eq:weak_ocp}--\eqref{eq:weak_state_eq}, then $j''(\bar{u})h^{2} \geq 0$ for all $h \in C_{\bar{u}}$. Conversely, if $(\bar{y},\bar{p},\bar{u})$ satisfies the first order optimality conditions \eqref{eq:weak_semi_bilinear_eq}, \eqref{eq:weak_adj_eq}, and \eqref{eq:char_var_ineq} and $j''(\bar{u})h^2>0$ for all $h$ in $C_{\bar{u}} \setminus \left\lbrace 0\right\rbrace$, then there exist $\mu >0$ and $\sigma >0$ such that
\begin{equation}\label{eq:local_cuad_growth_j}
    j(u) \geq j(\bar{u}) + \frac{\mu}{2}\| u - \bar{u}\|^{2}_{L^{2}(\Omega)} \qquad \forall u \in \mathbb{U}_{ad}: \| u - \bar{u} \|_{L^{2}(\Omega)} \leq \sigma.
\end{equation}
In particular, $\bar{u} \in \mathbb{U}_{ad}$ is a locally optimal control for \eqref{eq:weak_ocp}--\eqref{eq:weak_state_eq} in the sense of $L^2(\Omega)$.
\end{theorem}
\begin{proof}
The proof of the second order necessary optimality condition, if $\bar{u} \in \mathbb{U}_{ad}$ is a locally optimal control for \eqref{eq:weak_ocp}--\eqref{eq:weak_state_eq}, then $j''(\bar{u})h^{2} \geq 0$ for all $h \in C_{\bar{u}}$, follows from the arguments developed in \cite[Theorem 23, item \emph{i)}]{inbook}.


\EO{We now present a proof of the sufficient condition using a contradiction argument. Assume that \eqref{eq:local_cuad_growth_j} does not hold: For every $k \in \mathbb{N}$ there exists $u_{k} \in \mathbb{U}_{ad}$ such that}
\begin{equation}\label{eq:cont_second_order}
\| u_{k} - \bar{u} \|_{L^{2}(\Omega)} < k^{-1}, \qquad j(u_{k}) < j(\bar{u}) + (2k)^{-1} \| u_{k} - \bar{u} \|_{L^{2}(\Omega)}^{2}.
\end{equation}
Define, $\rho_{k}:= \|u_{k} - \bar{u}\|_{L^{2}(\Omega)}$ and $h_{k}:= \rho_{k}^{-1}(u_{k} - \bar{u})$ ($k \in \mathbb{N}$).
There exists a nonrelabeled subsequence such that $h_{k} \rightharpoonup h$ in $L^{2}(\Omega)$ as $k \uparrow \infty$. We now proceed in three steps.

\emph{Step 1.} We prove that $h \in C_{\bar{u}}$.
\EO{First, we note that $h$ satisfies \eqref{eq:sign_cond}; see \cite[Theorem 23, item \emph{ii)}]{inbook}.}
It remains to prove that $h(x) = 0$ if $\mathfrak{\bar{p}}(x) \neq 0$ for a.e.~$x \in \Omega$. To do this, we use the mean value theorem and \eqref{eq:cont_second_order} and obtain
\begin{equation}\label{eq:cont_second_order_2}
j'(\tilde{u}_{k})h_{k} = \rho_{k}^{-1}(j(u_{k})-j(\bar{u})) < \tfrac{\rho_{k}}{2k} \rightarrow 0 \quad k \uparrow \infty,
\qquad
\tilde{u}_{k} = \bar{u} + \theta_{k}(u_{k}-\bar{u}),
\end{equation}
where $\theta_{k} \in (0,1)$. We note that for every $k\in \mathbb{N}$, it holds that $\tilde{u}_{k} \in \mathbb{U}_{ad}$. From Lemma \ref{lem:diff_j}, it follows that $j'(\tilde{u}_{k})h_{k} = (\alpha \tilde{u}_{k}-\tilde{y}_{k}\tilde{p}_{k},h_{k})_{L^{2}(\Omega)}$ for every $k \in \mathbb{N}$, where $\tilde{y}_{k}=\mathcal{S}(\tilde{u}_{k})$ and $\tilde{p}_{k} = \Phi(\tilde{u}_{k})$. In the following we prove that $j'(\tilde{u}_{k})h_{k} \rightarrow j'(\bar{u})h$ as $k \uparrow \infty$. We start by noting that $\bar{y}-\tilde{y}_{k}$ is such that
\begin{multline}
\bar{y}-\tilde{y}_{k} \in H_{0}^{1}(\Omega):
\quad (\nabla (\bar{y}-\tilde{y}_{k}), \nabla v)_{L^{2}(\Omega)} + (\Theta(\cdot,\mathtt{y}_k)(\bar{y}-\tilde{y}_{k}),v)_{L^{2}(\Omega)}
\\
+ (\bar{u}(\bar{y}-\tilde{y}_{k}),v)_{L^{2}(\Omega)} 
 = -(\tilde{y}_{k}(\bar{u}-\tilde{u}_{k}),v)_{L^{2}(\Omega)}  \quad \forall v \in H^{1}_{0}(\Omega).
 \label{eq:y-tildeyk}
\end{multline}
Here, for every $k \in \mathbb{N}$, $\mathtt{y}_k := \bar{y} + \tilde{\theta}_k (\tilde{y}_k - \bar{y})$, where $\tilde{\theta}_{k} \in (0,1)$, and $\Theta(\cdot,\mathtt{z}) = \partial a/\partial \mathtt{z} (\cdot,\mathtt{z})$ a.e.~in $\Omega$ and for all $\mathtt{z} \in \mathbb{R}$. We note that, since $\bar{u} \in \mathbb{U}_{ad}$, we have $\bar{u} \in \mathcal{A}_0$; see assumption \ref{eq:restriction_a}. \EO{We can therefore apply the Lax-Milgram lemma and the arguments developed in the proof of Theorem \ref{thm:cont_reg} to deduce that}
\begin{equation}\label{eq:estim_y-tildey_k_H01_Linf}
\|\nabla(\bar{y}-\tilde{y}_{k})\|_{L^{2}(\Omega)} + \| \bar{y}-\tilde{y}_{k}\|_{C^{0,\varsigma}(\bar{\Omega})} 
\lesssim
\|\bar{u} - \tilde{u}_k \|_{L^2(\Omega)}.
\end{equation}
To derive the bound in $C^{0,\varsigma}(\bar{\Omega})$ we have used that $\| \bar{y}-\tilde{y}_{k}\|_{L^{\infty}(\Omega)} \lesssim \| \tilde{y}_{k} (\bar{u} - \tilde{u}_k)\|_{L^q(\Omega)} \leq \| \tilde{y}_{k}\|_{L^{\infty}(\Omega)} \| \bar{u} - \tilde{u}_k\|_{L^2(\Omega)} \lesssim \| f - a(\cdot,0)\|_{L^{q}(\Omega)} \| \bar{u} - \tilde{u}_k\|_{L^2(\Omega)}$, where the hidden constants are independent of $k$. We now use that $\|  \bar{u} - \tilde{u}_k \|_{L^2(\Omega)} \rightarrow 0$ as $k \uparrow \infty$, which follows from \eqref{eq:cont_second_order}, to deduce that
$\tilde{y}_{k} \rightarrow \bar{y}$ in  $H_{0}^{1}(\Omega) \cap C^{0,\varsigma}(\bar{\Omega})$ as $k \uparrow \infty$. Similarly, using Theorem \ref{thm:adjoint_problem} we can conclude that
\begin{equation}\label{eq:estim_p_ptilde_k}
  \| \nabla(\bar{p}-\tilde{p}_{k} ) \|_{L^r(\Omega)} 
  \lesssim 
\|  \bar{u} - \tilde{u}_k \|_{L^2(\Omega)} \rightarrow 0
\end{equation}   
as $k \uparrow \infty$. The convergence properties \eqref{eq:estim_y-tildey_k_H01_Linf} and \eqref{eq:estim_p_ptilde_k} allow us to conclude that $\mathfrak{\tilde{p}}_{k} := \alpha \tilde{u}_{k}-\tilde{y}_{k}\tilde{p}_{k}  \rightarrow \mathfrak{\bar{p}} = \alpha \bar{u}- \bar{y} \bar{p}$ in $L^{2}(\Omega)$ as $k \uparrow \infty$. This and $h_{k} \rightharpoonup h$ in $L^{2}(\Omega)$ as $k \uparrow \infty$ show that $j'(\tilde{u}_{k})h_{k} \rightarrow j'(\bar{u})h$ as $k \uparrow \infty$.  We therefore use \eqref{eq:cont_second_order_2} and obtain \EO{$j'(\bar{u})h \leq 0$; see \cite[Theorem 23, item \emph{ii)}]{inbook} for details.}
\EO{The remainder of the proof follows \DQ{from} \cite[Theorem 23, item \emph{ii)}]{inbook}.}
Thus, we have $h \in C_{\bar{u}}$.
 
\emph{Step 2.} We prove that $h \equiv 0$. In a first step, we use Taylor's theorem, $j'(\bar{u})h_{k} \geq 0$ for every $k \in \mathbb{N}$, and the estimate on the right hand side of \eqref{eq:cont_second_order} to obtain
\begin{equation}
    \tfrac{\rho_{k}^{2}}{2}j''(\hat{u}_{k})h_{k}^{2} = j(u_{k})-j(\bar{u}) - j'(\bar{u})(u_{k}-\bar{u}) < \tfrac{\rho_{k}^{2}}{2k}, \qquad k \in \mathbb{N},  
\end{equation}
where $\hat{u}_{k} := \bar{u} + \theta_{k}(u_{k}-\bar{u})$ with $\theta_{k} \in (0,1)$. Therefore, $\liminf_{k \uparrow \infty}j''(\hat{u}_{k})h_{k}^{2} \leq 0$. 

In the following, we prove that $j''(\bar{u})h^{2} \leq \liminf_{k \uparrow \infty}j''(\hat{u}_{k})h^{2}_{k}$. For this purpose, we first use the characterization of $j''$ from Lemma \ref{lem:diff_j} and write
\begin{align}\label{eq:thm_3.9_J_functional_weak}
\begin{split}
	j''(\hat{u}_{k})h_{k}^{2}  = \alpha\|h_{k}\|^{2}_{L ^{2}(\Omega)} - 
	\EO{(\hat{z}_{k}\hat{p}_{k} + \hat{y}_{k}\hat{\eta}_{k},h_{k})_{L^2(\Omega)}.}
 \end{split}
\end{align}
Here, $\hat{y}_{k} = \mathcal{S}(\hat{u}_{k})$, $\hat{z}_{k} =\mathcal{S}'(\hat{u}_{k})h_{k}$, $\hat{p}_{k} = \Phi(\hat{u}_{k})$, and  $\hat{\eta}_{k} = \Phi'(\hat{u}_{k})h_{k}$, where $\mathcal{S}$ and $\mathcal{S}'$ are as in the statement of Theorem \ref{thm:properties_S} and $\Phi$ and $\Phi'$ are as in the statement of Theorem \ref{thm:diff_u_p}. \EO{Let us now analyze the convergence of $\{ \hat{y}_{k}\}_{k \in \mathbb{N}}$, $\{ \hat{z}_{k}\}_{k \in \mathbb{N}}$, $\{ \hat{p}_{k}\}_{k \in \mathbb{N}}$, and $\{ \hat{\eta}_{k}\}_{k \in \mathbb{N}}$.}

In a first step, we analyze the convergence of $\{ \hat{y}_{k}\}_{k \in \mathbb{N}}$ and $\{ \hat{p}_{k}\}_{k \in \mathbb{N}}$. Using similar arguments as in the derivation of \eqref{eq:estim_y-tildey_k_H01_Linf} and \eqref{eq:estim_p_ptilde_k} in \emph{Step 1}, we obtain that
\begin{equation}\label{eq:estim_y-haty_k_estim_p_hatp_k}
  \|\nabla(\bar{y}-\hat{y}_{k})\|_{L^{2}(\Omega)} 
  + 
  \| \bar{y}-\hat{y}_{k}\|_{C^{0,\varsigma}(\bar{\Omega})} 
  +
  \| \nabla(\bar{p} - \hat{p}_{k} ) \|_{L^r(\Omega)} 
\lesssim
\|\bar{u} - \hat{u}_k \|_{L^2(\Omega)}.
\end{equation}
Thus, $\hat{y}_{k} \rightarrow \bar{y}$ in $H_{0}^{1}(\Omega) \cap C^{0,\varsigma}(\bar{\Omega})$ and $\hat{p}_{k} \rightarrow \bar{p}$ in $W^{1,r}_{0}(\Omega)$ as $k \uparrow \infty$, because $\| \bar{u} - \hat{u}_k \|_{L^2(\Omega)} \leq \| \bar{u} - u_k \|_{L^2(\Omega)} \rightarrow 0$ as $k \uparrow \infty$ (see the bound in the left hand side of \eqref{eq:cont_second_order}).

In a second step, we prove that $\hat{z}_{k} \rightarrow \bar{z} = \mathcal{S}'(\bar{u})h$ in $H_{0}^{1}(\Omega) \cap C^{0,\varsigma}(\bar{\Omega})$ as $k \uparrow \infty$. To do this, we note that $\bar{z} - \hat{z}_{k} \in H_{0}^{1}(\Omega)$ solves the problem
\begin{multline*}
(\nabla (\bar{z} - \hat{z}_{k}),\nabla v)_{L^{2}(\Omega)} + \left(\left[ \tfrac{\partial a}{\partial y}(\cdot, \bar{y}) + \bar{u} \right](\bar{z} - \hat{z}_{k}), v \right)_{L^{2}(\Omega)} 
 = -((h- h_{k})\bar{y},v)_{L^{2}(\Omega)} 
 \\
 - (h_{k}(\bar{y} - \hat{y}_{k}),v)_{L^{2}(\Omega)} 
 - \left( \left[ \tfrac{\partial a}{\partial y}(\cdot,\bar{y}) - \tfrac{\partial a}{\partial y}(\cdot,\hat{y}_{k}) \right] \hat{z}_{k},v \right)_{L^{2}(\Omega)} 
 - \left((\bar{u} - \hat{u}_{k}) \hat{z}_{k},v \right)_{L^{2}(\Omega)}
\end{multline*}
for all $v \in H_0^1(\Omega)$. Note that, by definition, $\| h_k \|_{L^2(\Omega)} = 1$ and that
$
\|\nabla \hat{z}_{k}\|_{L^{2}(\Omega)} \lesssim \|f - a(\cdot,0)\|_{H^{-1}(\Omega)}
$
for all $k \in \mathbb{N}$. If we set $v = \bar{z} - \hat{z}_{k}$ in the previous problem and use that $\bar{u} \in \mathcal{A}_0$, the uniform boundedness of the sequences $\{h_{k}\}_{k \in \mathbb{N}}$ and $\{\hat{z}_{k}\}_{k \in \mathbb{N}}$ in $L^{2}(\Omega)$ and $H_{0}^{1}(\Omega)$, respectively, and the bound \eqref{eq:estim_y-haty_k_estim_p_hatp_k}, we can obtain the \EO{estimate
$
\|\nabla (\bar{z} - \hat{z}_{k}) \|_{L^{2}(\Omega)} 
\lesssim 
\|h - h_{k} \|_{H^{-1}(\Omega)}
+
\| \bar{u} - \hat{u}_{k} \|_{L^{2}(\Omega)}.
$
We} now prove that $\hat{z}_k \rightarrow \bar{z}$ in $W^{1,\kappa}(\Omega)$ as $k \uparrow \infty$, where $\kappa > 4$ if $d = 2$ and $\kappa > 3$ if $d = 3$.
\EO{We proceed using a bootstrap argument and \cite[Theorem 0.5]{MR1331981}: First, we prove that $\bar{z} - \hat{z}_{k} \in H_0^1(\Omega) \cap W^{1,3}(\Omega)$ and then that $\bar{z} - \hat{z}_k \in H_0^1(\Omega) \cap W^{1,\kappa}(\Omega)$ with}
\begin{equation}
\|\nabla (\bar{z} - \hat{z}_{k}) \|_{L^{\kappa}(\Omega)} 
\lesssim
\|h - h_{k} \|_{W^{-1,\kappa}(\Omega)} 
+
\| \bar{u} - \hat{u}_{k} \|_{L^{2}(\Omega)}.
\label{eq:z-zkW1kappa}
\end{equation}
The last step is to use that $L^2(\Omega) \hookrightarrow W^{-1,\kappa}(\Omega)$ is compact in order to deduce that $\hat{z}_{k}\rightarrow \bar{z}$ in $H_0^1(\Omega) \cap W^{1,\kappa}(\Omega)$ as $k \uparrow \infty$. This implies that $\hat{z}_{k}\rightarrow \bar{z}$ in $C(\bar \Omega)$ as $k \uparrow \infty$.

We now prove that $\hat{\eta}_{k} \rightarrow \bar{\eta} = \Phi'(\bar{u})h $ in $W^{1,r}_{0}(\Omega)$ as $k \uparrow \infty$. To this end, 
we apply a stability bound for the problem that solves $\bar{\eta}-\hat{\eta}_{k} \in W^{1,r}_{0}(\Omega)$ and obtain
\begin{multline}\label{eq:stability_eta_hat_eta}
\|\nabla(\bar{\eta}-\hat{\eta}_{k})\|_{L^{r}(\Omega)} \lesssim \sum_{t \in \mathcal{D}}|\bar{z}(t)-\hat{z}_{k}(t)| + 
\left\| \left[\tfrac{\partial a}{\partial y}(\cdot,\bar{y}) -  \tfrac{\partial a}{\partial y}(\cdot,\hat{y}_{k})\right]\hat{\eta}_{k} \right\|_{W^{-1,r}(\Omega)} 
\\
+ 
\left\| (\bar{u}-\hat{u}_{k})\hat{\eta}_{k} \right\|_{W^{-1,r}(\Omega)} 
+ 
\left\| \left[\tfrac{\partial^{2} a}{\partial y^{2}}(\cdot, \bar{y})\bar{z} + h\right]\bar{p} - \left[\tfrac{\partial^{2} a}{\partial y^{2}}(\cdot, \hat{y}_{k})\hat{z}_{k} + h_{k}\right]\hat{p}_{k} \right\|_{W^{-1,r}(\Omega)} 
\\
:= \mathbf{I}_{k} + \mathbf{II}_{k} +\mathbf{III}_{k} + \mathbf{IV}_{k}.
\end{multline}
The control of $\mathbf{I}_{k}$ follows directly from \eqref{eq:z-zkW1kappa}: $\mathbf{I}_{k} \lesssim \|\bar{z}-\hat{z}_{k}\|_{C(\bar{\Omega})} \lesssim \|\nabla (\bar{z} - \hat{z}_{k}) \|_{L^{\kappa}(\Omega)}$. This implies that $\mathbf{I}_{k} \rightarrow 0$ as $k \uparrow \infty$. We now proceed with the analysis of $\mathbf{II}_{k}$ and $\mathbf{III}_{k}$. To this end, we first note that the uniform boundedness of $\{ \hat{y}_k \}_{k \in \mathbb{N}}$, $\{ \hat{z}_k \}_{k \in \mathbb{N}}$, and $\{ \hat{p}_k \}_{k \in \mathbb{N}}$ in $L^{\infty}(\Omega)$, $L^{\infty}(\Omega)$, and $W_0^{1,r}(\Omega)$, respectively, allow us to conclude that $\{ \hat{\eta}_k \}_{k \in \mathbb{N}}$ is uniformly bounded in $W_0^{1,r}(\Omega)$. This, the bound \eqref{eq:estim_y-haty_k_estim_p_hatp_k}, and the fact that $\hat{u}_{k} \rightarrow \bar{u}$ in $L^{2}(\Omega)$ imply that $\mathbf{II}_{k}\rightarrow 0$ as $k \uparrow \infty$. Similarly, $\mathbf{III}_{k} \rightarrow 0$ as $k \uparrow \infty$. It remains to examine $\mathbf{IV}_{k}$. Straightforward calculations show that
\begin{multline*}
\mathbf{IV}_{k} \leq 
\left\| \left[\tfrac{\partial ^{2}a}{\partial y ^{2}}(\cdot, \bar{y})-\tfrac{\partial ^{2}a}{\partial y ^{2}}(\cdot, \hat{y}_{k}) \right] \bar{z} \bar{p} \right\|_{W^{-1,r}(\Omega)} 
+ 
\left\| \tfrac{\partial^{2}a}{\partial y^{2}}(\cdot,\hat{y}_k) (\bar{z} - \hat{z}_{k}) \bar{p} \right\|_{W^{-1,r}(\Omega)} \\ 
+ 
\left\| \tfrac{\partial^{2}a}{\partial y^{2}}(\cdot,\hat{y}_k)\hat{z}_{k} ( \bar{p} - \hat{p}_k) \right\|_{W^{-1,r}(\Omega)}
+
\|h(\bar{p}-\hat{p}_{k})\|_{W^{-1,r}(\Omega)} 
+
\|(h-h_{k})\hat{p}_{k}\|_{W^{-1,r}(\Omega)}.
\end{multline*}
Similarly, we estimate the terms that appear on the right hand side of the previous inequality. In fact, we use the bounds \eqref{eq:estim_y-haty_k_estim_p_hatp_k} and \EO{\eqref{eq:z-zkW1kappa}}, the uniform boundedness of the sequences $\{ \hat{y}_{k}  \}_{k \in \mathbb{N}}$ and $\{ \hat{z}_{k}  \}_{k \in \mathbb{N}}$ in $L^{\infty}(\Omega)$ and $L^2(\Omega)$, respectively, 
\EO{$h_k \rightharpoonup h$ in $L^2(\Omega)$, and 
$\hat{u}_k \rightarrow \bar{u}$ in $L^2(\Omega)$ as $k \uparrow \infty$} in order to deduce that $\mathbf{IV}_{k} \rightarrow 0$ as $k \uparrow \infty$. Consequently, we can deduce that $\hat{\eta}_{k} \rightarrow \bar{\eta}$ in $W^{1,r}_{0}(\Omega)$ as $k \uparrow \infty$.

At this point, we have thus concluded that $\hat{y}_{k}\rightarrow \bar{y}$ in $H_0^{1}(\Omega) \cap C^{0,\varsigma}(\bar \Omega)$, $\hat{p}_{k}\rightarrow \bar{p}$ in $W_0^{1,r}(\Omega)$, $\hat{z}_{k}\rightarrow \bar{z}$ in $W_0^{1,\kappa}(\Omega)$, and that $\hat{\eta}_{k} \rightarrow \bar{\eta}$ in $W^{1,r}_{0}(\Omega)$ as $k \uparrow \infty$. \EO{We thus proceed as in \cite[Theorem 23, item \emph{ii)}]{inbook} and deduce that $j''(\bar{u})h^{2} \leq 0$.}
Since $j''(\bar{u})v^{2} >0$ for all $v \in C_{\bar{u}} \setminus \{0\}$ and $h \in C_{\bar{u}}$, we conclude that $h \equiv 0$.

\emph{Step 3}. Contradiction. \EO{Proceeding as in \cite[Theorem 23, item \emph{ii)}]{inbook},
we conclude that $\alpha \leq 0$. This contradicts the fact that $\alpha>0$ and completes the proof.}
\end{proof}

Define $C_{\bar{u}}^{\tau}:=\left\lbrace h \in L^{2}(\Omega) \text{ satisfying \eqref{eq:sign_cond} and } h(x) = 0 \text{ if }  |\bar{\mathfrak{p}}(x)|>\tau\right\rbrace$. We conclude this section with the following equivalence result.

\begin{theorem}[equivalent optimal conditions]
If $(\bar{y},\bar{p},\bar{u})$ satisfies the first order conditions \eqref{eq:weak_semi_bilinear_eq}, \eqref{eq:weak_adj_eq}, and \eqref{eq:char_var_ineq}, then the following statements are equivalent:
 \begin{equation}\label{eq:equiv_cond_1}
     j''(\bar{u})h^{2} > 0 \quad \forall h \in C_{\bar{u}} \setminus \left\lbrace 0 \right\rbrace
 \end{equation}
 and 
 \begin{equation}\label{eq:equiv_cond_2}
     \exists \mu,\tau > 0: \quad j''(\bar{u})h^{2} \geq \mu \|h\|^{2}_{L^{2}(\Omega)} \quad \forall h \in C^{\tau}_{\bar{u}}.
 \end{equation}
\end{theorem}
\begin{proof}
The proof essentially follows the same arguments as in the proof of \cite[Theorem 25]{inbook}, in combination with the ideas developed in the proof of Theorem \ref{thm:second_order_cond}. For the sake of brevity, we omit details.
\end{proof}


\section{Regularity properties of an optimal control}\label{sec:loc_opt_control_reg}

In this section, we discuss regularity properties of a locally optimal control for the problem \eqref{eq:weak_ocp}--\eqref{eq:weak_state_eq}.


\subsection{\EO{$H^{1}(\Omega)$-}regularity on Lipschitz domains}

Let $\Pi_{[\mathtt{a},\mathtt{b}]}: L^{1}(\Omega) \rightarrow \mathbb{U}_{ad}$ be the nonlinear operator defined by $\Pi_{[\mathtt{a},\mathtt{b}]}(v):= \min\{\mathtt{b},\max\{\mathtt{a},v\}\}$ a.e.~in $\Omega$. \EO{We note that $\Pi_{[\mathtt{a},\mathtt{b}]}$ can be written equivalently as follows:}
\begin{equation}\label{eq:equiv_proj_form}
\Pi_{[\mathtt{a},\mathtt{b}]}(v) = v + \max\{0,\mathtt{a} - v\} - \max\{0,-\mathtt{b}+v\}.
\end{equation}
With the nonlinear operator $\Pi_{[\mathtt{a},\mathtt{b}]}$ in hand, we have the following projection formula \cite[section 4.6]{MR2583281}: If $\bar{u}$ denotes a locally optimal control for problem \eqref{eq:weak_ocp}--\eqref{eq:weak_state_eq}, then
\begin{equation}\label{eq:u_proj_form}
\bar{u}(x) = \Pi_{[\mathtt{a},\mathtt{b}]}\left(\alpha^{-1}\bar{y}(x)\bar{p}(x) \right) \quad \text{for a.e. } x \in \Omega.
\end{equation}

\EO{We now provide an $H^{1}(\Omega)$-regularity result for a locally optimal control $\bar{u}$.}

\begin{theorem}[$H^{1}(\Omega)$-regularity of $\bar{u}$]
\label{thm:H1-reg_u} 
If $\bar{u}$ is a locally optimal control, \EO{then $\bar{u} \in H^{1}(\Omega)$.}
\end{theorem}
\begin{proof}
\EO{Define $\phi := \alpha^{-1}\bar{y}\bar{p}$. By suitable applications of H\"older's inequality and using the regularity properties $\bar{y} \in H_{0}^{1}(\Omega) \cap W^{1,\kappa}(\Omega)$ (cf.~Theorem \ref{thm:cont_reg}) and $\bar{p} \in W_{0}^{1,r}(\Omega)$ for \DQ{every} $r< d/(d-1)$ such that \eqref{eq:r} holds, it follows that $\phi \in W_{0}^{1,r}(\Omega)$. Then, $-\Delta \phi$ can be identified with a measure $\lambda$ in $\Omega$ \cite{MR0192177,MR0812624}. With these elements, the desired regularity of $\bar{u}$ follows from the projection formula \eqref{eq:u_proj_form} and \cite[Lemma 3.3]{MR3264224}. This concludes the proof}.
\end{proof}

\subsection{Additional regularity properties on convex \EO{polytopal domains}}

\EO{In the following, we assume that $\Omega$ is a convex polytope and that the data are sufficiently regular. We then obtain an additional regularity result for a locally optimal control.}


\begin{theorem}[local/global Lipschitz regularity of $\bar{u}$]
\label{thm:C01_reg_baru}
Let us assume that $\Omega$ is \EO{a} convex \EO{ polytopal domain}. Let $\bar{u}$ be a locally optimal control. If $f,a(\cdot,0) \in L^{\eta}(\Omega)$, for some $\eta>d$, then $\bar{u} \in W^{1,\infty}(\Omega \setminus \cup_{t \in \mathcal{E}} \bar{B}_{t}(\rho_t))$, where $B_{t}(\rho_t) \subset \Omega$ denotes an open ball with center $t$ and a positive radius $\rho_t$. Moveover, if $d=2$ and $\mathtt{a}>0$, then 
$\bar{u} \in C^{0,1}(\bar{\Omega})$.
\end{theorem}
\begin{proof}
For the sake of simplicity, we assume that $\mathcal{D} = \mathcal{E} = \{t\}$ and omit the subindex $t$ in  $\rho_t$. \EO{We recall that $\mathcal{E}$ is defined in \eqref{eq:set_mathcal_E}. The case $t \not\in \mathcal{E}$ is discussed in Remark \ref{rem:t_notin_E} below.} 

\EO{First, we use \eqref{eq:equiv_proj_form} and \cite[Theorem A.1]{MR1786735} to conclude that $\nabla \bar{u} = \alpha^{-1}\nabla(\bar{y}\bar{p}) \chi$ in the sense of distributions, where $\chi$ denotes the characteristic function of the set $\{x \in \Omega : \mathtt{a} < \alpha^{-1}\bar{y}(x)\bar{p}(x) < \mathtt{b} \}$. Taking this formula into account, we obtain}
\begin{equation}\label{eq:Omega_estimate_Linf}
\| \nabla \bar{u}\|_{L^{\infty}(\Omega)} \leq \alpha^{-1} \left( \| \nabla(\bar{y}\bar{p})\chi\|_{L^{\infty}(\Omega \setminus \bar{B}_{t}(\rho))} + \|\nabla(\bar{y}\bar{p})\chi\|_{L^{\infty}(B_{t}(\rho))} \right),
\end{equation}
\EO{where $B_{t}(\rho)$ is as in the statement of the theorem. Define $\mathfrak{I}:= \| \nabla(\bar{y}\bar{p})\chi\|_{L^{\infty}(\Omega \setminus \bar{B}_{t}(\rho))}$. To control $\mathfrak{I}$, we use the fact that $\Omega$ is a convex polytope and that $f,a(\cdot,0) \in L^{\eta}(\Omega)$, for some $\eta > d$, and apply the regularity results of \cite[Lemma 4.1]{MR3973329} and Theorem \ref{thm:H^2-reg-p} to obtain}
\[
\EO{ \bar{y} \in W^{1,\infty}(\Omega),
 \qquad
 \bar{p}\in H^2(\Omega \setminus \bar{B}_{t}(\rho)) \cap C^{0,1}(\bar{\Omega} \setminus B_{t}(\rho)).}
\]
\EO{Consequently,}
\begin{equation}\label{eq:bound_mathfrakI}
\mathfrak{I} \leq \|\bar{p}\|_{L^{\infty}(\Omega \setminus \bar{B}_{t}(\rho))}\|\nabla \bar{y}\|_{L^{\infty}(\Omega)} + \|\bar{y}\|_{L^{\infty}(\Omega)}\|\nabla \bar{p}\|_{L^{\infty}(\Omega \setminus \bar{B}_{t}(\rho))} < \infty.
\end{equation}
This shows that $\bar u \in W^{1,\infty}(\Omega \setminus \bar{B}_{t}(\rho ))$ in both two and three dimensions and proves the first assertion of the theorem.

\EO{We now bound $\|\nabla(\bar{y}\bar{p})\chi\|_{L^{\infty}(B_{t}(\rho))}$ in \eqref{eq:Omega_estimate_Linf}. We proceed in two cases.}

\EO{\emph{Case 1.} $\bar{y}(t) \neq 0$. Given that $t \in \mathcal{E}$, we have that $0 \neq \bar{y}(t) \neq y_t$. We now use that $\bar{y} \in C(\bar \Omega)$ (cf.~Theorem \ref{thm:W1kappa-regularity}) to conclude that there exists $\varepsilon > 0$ such that $\bar{y} \neq 0$ in $\bar{B}_{t}(\varepsilon)$. Assume that $\varepsilon \in (0,\rho)$. 
We thus have that $\zeta:= \min\left\lbrace |\bar{y}(x)|:x \in \bar{B}_{t}(\varepsilon)\right\rbrace > 0$.  Define $\mathtt{c}:= \max\{|\mathtt{a}|,|\mathtt{b}|\}$ and let $M>\alpha\zeta^{-1}\mathtt{c}$. An application of \cite[Lemma 5.1] {MR3973329} shows that there exists $\delta$ such that $|\bar{p}(x)| \geq M$ for all $x \in \bar{B}_{t}(\delta)$. Assume that $\delta \in (0,\varepsilon)$. From these properties, we can thus deduce that
\[
|\alpha^{-1}\bar{y}(x)\bar{p}(x)| \geq \alpha^{-1}\min\left\lbrace |\bar{y}(x)|:x \in \bar{B}_{t}(\delta)\right\rbrace|\bar{p}(x)| > \alpha^{-1} \zeta (\alpha\zeta^{-1}\mathtt{c}) = \mathtt{c} \quad \forall x \in \bar{B}_{t}(\delta).
\]
Since, by definition, $\mathtt{c} = \max\{|\mathtt{a}|,|\mathtt{b}|\}$, we can conclude that the ball $\bar{B}_{t}(\delta)$ is such that $\chi \equiv 0$ in $\bar{B}_{t}(\delta)$. Recall that $\chi$ is the characteristic function of $\{x \in \Omega : \mathtt{a} < \alpha^{-1}\bar{y}(x)\bar{p}(x) < \mathtt{b} \}$. We thus proceed as follows:
\[
\|\nabla(\bar{y}\bar{p})\chi\|_{L^{\infty}(B_{t}(\rho))} = \|\nabla(\bar{y}\bar{p})\chi\|_{L^{\infty}(B_{t}(\rho) \setminus \bar{B}_{t}(\delta))} \leq
\|\nabla(\bar{y}\bar{p})\chi\|_{L^{\infty}(\Omega \setminus \bar{B}_{t}(\delta))} < \infty.
\]
To control the last term, we have used the regularity results of \cite[Lemma 4.1]{MR3973329} and Theorem \ref{thm:H^2-reg-p} as in the derivation of \eqref{eq:bound_mathfrakI}.}

\EO{\emph{Case 2.} $\bar{y}(t) = 0$. Given that $t \in \mathcal{E}$, we have $0 = \bar{y}(t) \neq y_t$. In this case, we assume $d = 2$ and $\mathtt{a} > 0$. We begin the analysis with the following asymptotic behavior of $\bar{p}$ near $t$ \cite{MR161019,MR657523,MR1156467,MR3169756,MR3941634,MR431753}:}
\begin{equation}\label{eq:asymp_behav_p_2}
\EO{|\bar{p}(x)| \lesssim |\log |x-t|| + 1.}
\end{equation}
We now use that $\bar{y}(t) = 0$, $\bar{y} \in W^{1,\infty}(\Omega)$, and that $W^{1,\infty}(\Omega) = C^{0,1}(\bar \Omega)$ \cite[Theorem 4.1]{MR2177410} to deduce that
\begin{equation}\label{eq:lim_1_C_reg}
|\bar{y}(x)\bar{p}(x)| = \EO{|(\bar{y}(x) - \bar{y}(t)) \bar{p}(x)|} \lesssim |x-t| \cdot (|\log |x-t||+1) \rightarrow 0,  \quad x \rightarrow t.
\end{equation}
Since $\mathtt{a}>0$, this shows the existence of $\delta>0$ such that $\alpha^{-1}|\bar{y}(x)\bar{p}(x)| < \mathtt{a}$ for all $x \in B_t(\delta)$. Consequently, $B_t(\delta)$ is such that $\chi \equiv 0$ in $B_t(\delta)$ and thus $\|\nabla(\bar{y}\bar{p})\chi\|_{L^{\infty}(B_{t}(\delta))} = 0$. Assume that $\delta \in (0,\rho)$. The term $\|\nabla(\bar{y}\bar{p})\chi\|_{L^{\infty}(B_{t}(\rho) \setminus B_{t}(\delta))}$ is controlled as before. As a result, we obtain that $\|\nabla(\bar{y}\bar{p})\chi\|_{L^{\infty}(B_{t}(\rho))}$ is bounded.

Having deduced that $\|\nabla(\bar{y}\bar{p})\chi\|_{L^{\infty}(B_{t}(\rho))} < \infty$, we conclude that $\|\nabla(\bar{y}\bar{p})\chi\|_{L^{\infty}(\Omega)}$ is bounded. The fact that $W^{1,\infty}(\Omega) = C^{0,1}(\bar \Omega)$ allows the conclusion to be drawn.
\end{proof}

\begin{remark}[$t \notin \mathcal{E}$]\label{rem:t_notin_E}
\EO{Let us illustrate the case where $t \not\in \mathcal{E}$ with two specific examples: $\mathcal{D} = \{ t \}$ and $\mathcal{D} = \{ t_1 ,t_2 \}$.}
\begin{itemize}
\item[(i)] \EO{Let $\mathcal{D} = \{ t \}$, where $t$ is such that $t \notin \mathcal{E}$, i.e., $\bar{y}(t) = y_t$. Consequently, $\bar{p} \equiv 0$ in $\Omega$ and $\bar{u}$ is a constant in $\Omega$.}
\item[(ii)] \EO{Let $\mathcal{D} = \{ t_1 ,t_2 \}$, where $t_1 \in \mathcal{E}$ and $t_2 \notin \mathcal{E}$, i.e., $\bar{y}(t_1) \neq y_{t_1}$ and $\bar{y}(t_2) = y_{t_2}$. In this setting, we can deduce that
\[
\bar{y} \in W^{1,\infty}(\Omega),
\qquad
\bar{p}\in H^2(\Omega \setminus \bar{B}_{t_1}(\rho_1)) \cap C^{0,1}(\bar{\Omega} \setminus B_{t_1}(\rho_1)).
\]
As a result, $\bar{y} \bar{p} \in C^{0,1}(\bar{B}_{t_2}(\rho_2))$ and $\bar{u}$ is regular in a neighborhood of $t_2$.}
\end{itemize}
\end{remark}

\bigskip

\textbf{Data availability.} No data was used for the research described in the article.

\bigskip

\EO{\textbf{Acknowledgments.} We would like to express our sincere gratitude to the anonymous referees and the editor for providing insightful comments and suggestions, which led to better results and an improved presentation.
}

\bibliographystyle{siamplain}
\bibliography{bil_track_ref}

\begin{thebibliography}{10}

\bibitem{adams2003sobolev}
{\sc R.~A. Adams and J.~J. Fournier}, {\em Sobolev spaces}, Elsevier, 2003.

\bibitem{MR4438718}
{\sc A.~Allendes, F.~Fuica, and E.~Ot\'{a}rola}, {\em Error estimates for a
  pointwise tracking optimal control problem of a semilinear elliptic
  equation}, SIAM J. Control Optim., 60 (2022), pp.~1763--1790,
  \url{https://doi.org/10.1137/20M1364151}.

\bibitem{MR3679932}
{\sc A.~Allendes, E.~Ot\'arola, R.~Rankin, and A.~J. Salgado}, {\em Adaptive
  finite element methods for an optimal control problem involving {D}irac
  measures}, Numer. Math., 137 (2017), pp.~159--197,
  \url{https://doi.org/10.1007/s00211-017-0867-9}.

\bibitem{MR3973329}
{\sc N.~Behringer, D.~Meidner, and B.~Vexler}, {\em Finite element error
  estimates for optimal control problems with pointwise tracking}, Pure Appl.
  Funct. Anal., 4 (2019), pp.~177--204.

\bibitem{MR735811}
{\sc D.~C. Bi and G.~T. Zhu}, {\em The optimal control of the power
  distribution of the nuclear reactor}, in Control science and technology for
  the progress of society, {V}ol. 1 ({K}yoto, 1981), IFAC, Laxenburg, 1982,
  pp.~225--230.

\bibitem{MR2349487}
{\sc L.~T. Biegler, O.~Ghattas, M.~Heinkenschloss, D.~Keyes, and B.~van
  Bloemen~Waanders}, eds., {\em Real-time {PDE}-constrained optimization},
  Society for Industrial and Applied Mathematics (SIAM), Philadelphia, PA,
  2007, \url{https://doi.org/10.1137/1.9780898718935}.

\bibitem{MR0414174}
{\sc C.~Bruni, G.~DiPillo, and G.~Koch}, {\em Bilinear systems: an appealing
  class of ``nearly linear'' systems in theory and applications}, IEEE Trans.
  Automatic Control, AC-19 (1974), pp.~334--348,
  \url{https://doi.org/10.1109/tac.1974.1100617}.

\bibitem{MR0812624}
{\sc E.~Casas}, {\em {$L^2$} estimates for the finite element method for the
  {D}irichlet problem with singular data}, Numer. Math., 47 (1985),
  pp.~627--632, \url{https://doi.org/10.1007/BF01389461}.

\bibitem{MR2434067}
{\sc E.~Casas}, {\em Necessary and sufficient optimality conditions for
  elliptic control problems with finitely many pointwise state constraints},
  ESAIM Control Optim. Calc. Var., 14 (2008), pp.~575--589,
  \url{https://doi.org/10.1051/cocv:2007063}.

\bibitem{MR4679130}
{\sc E.~Casas, K.~Chrysafinos, and M.~Mateos}, {\em Semismooth {N}ewton method
  for boundary bilinear control}, IEEE Control Syst. Lett., 7 (2023),
  pp.~3549--3554, \url{https://doi.org/10.1109/lcsys.2023.3337747}.

\bibitem{Casas_bilinear_1}
{\sc E.~Casas, K.~Chrysafinos, and M.~Mateos}, {\em Bilinear control of
  semilinear elliptic {PDE}s: convergence of a semismooth {N}ewton method},
  Numer. Math., 157 (2025), pp.~143--163,
  \url{https://doi.org/10.1007/s00211-024-01448-1}.

\bibitem{Casas_bilinear_2}
{\sc E.~Casas, K.~Chrysafinos, and M.~Mateos}, {\em Error estimates for the
  discretization of bilinear control problems governed by semilinear elliptic
  {PDE}s}, Math. Control Relat. Fields, 15 (2025), pp.~1320--1345,
  \url{https://doi.org/10.3934/mcrf.2024049}.

\bibitem{MR4599941}
{\sc E.~Casas, K.~Kunisch, and F.~Tr\"oltzsch}, {\em Optimal control of {PDE}s
  and {FE}-approximation}, in Numerical control. {P}art {A}, vol.~23 of Handb.
  Numer. Anal., North-Holland, Amsterdam, [2022] \copyright 2022, pp.~115--163.

\bibitem{MR3586845}
{\sc E.~Casas and M.~Mateos}, {\em Optimal control of partial differential
  equations}, in Computational mathematics, numerical analysis and
  applications, vol.~13 of SEMA SIMAI Springer Ser., Springer, Cham, 2017,
  pp.~3--59.

\bibitem{inbook}
{\sc E.~Casas and M.~Mateos}, {\em Optimal Control of Partial Differential
  Equations}, 08 2017, pp.~3--59,
  \url{https://doi.org/10.1007/978-3-319-49631-3_1}.

\bibitem{MR3264224}
{\sc E.~Casas, M.~Mateos, and B.~Vexler}, {\em New regularity results and
  improved error estimates for optimal control problems with state
  constraints}, ESAIM Control Optim. Calc. Var., 20 (2014), pp.~803--822,
  \url{https://doi.org/10.1051/cocv/2013084}.

\bibitem{MR2902693}
{\sc E.~Casas and F.~Tr\"oltzsch}, {\em Second order analysis for optimal
  control problems: improving results expected from abstract theory}, SIAM J.
  Optim., 22 (2012), pp.~261--279, \url{https://doi.org/10.1137/110840406},
  \url{https://doi.org/10.1137/110840406}.

\bibitem{MR3449612}
{\sc L.~Chang, W.~Gong, and N.~Yan}, {\em Numerical analysis for the
  approximation of optimal control problems with pointwise observations}, Math.
  Methods Appl. Sci., 38 (2015), pp.~4502--4520,
  \url{https://doi.org/10.1002/mma.2861}.

\bibitem{MR3308473}
{\sc J.~C. De~los Reyes}, {\em Numerical {PDE}-constrained optimization},
  SpringerBriefs in Optimization, Springer, Cham, 2015,
  \url{https://doi.org/10.1007/978-3-319-13395-9}.

\bibitem{MR2050138}
{\sc A.~Ern and J.-L. Guermond}, {\em Theory and practice of finite elements},
  vol.~159 of Applied Mathematical Sciences, Springer-Verlag, New York, 2004,
  \url{https://doi.org/10.1007/978-1-4757-4355-5}.

\bibitem{MR1156467}
{\sc S.~J. Fromm}, {\em Potential space estimates for {G}reen potentials in
  convex domains}, Proc. Amer. Math. Soc., 119 (1993), pp.~225--233,
  \url{https://doi.org/10.2307/2159846}.

\bibitem{MR4416986}
{\sc R.~Glowinski, Y.~Song, X.~Yuan, and H.~Yue}, {\em Bilinear optimal control
  of an advection-reaction-diffusion system}, SIAM Rev., 64 (2022),
  pp.~392--421, \url{https://doi.org/10.1137/21M1389778}.

\bibitem{MR3396210}
{\sc P.~Grisvard}, {\em Elliptic problems in nonsmooth domains}, vol.~69 of
  Classics in Applied Mathematics, Society for Industrial and Applied
  Mathematics (SIAM), Philadelphia, PA, 2011,
  \url{https://doi.org/10.1137/1.9781611972030.ch1}.

\bibitem{MR657523}
{\sc M.~Gr\"uter and K.-O. Widman}, {\em The {G}reen function for uniformly
  elliptic equations}, Manuscripta Math., 37 (1982), pp.~303--342,
  \url{https://doi.org/10.1007/BF01166225}.

\bibitem{MR2177410}
{\sc J.~Heinonen}, {\em Lectures on {L}ipschitz analysis}, vol.~100 of Report.
  University of Jyv\"askyl\"a{} Department of Mathematics and Statistics,
  University of Jyv\"askyl\"a, Jyv\"askyl\"a, 2005.

\bibitem{MR2516528}
{\sc M.~Hinze, R.~Pinnau, M.~Ulbrich, and S.~Ulbrich}, {\em Optimization with
  {PDE} constraints}, Springer, New York, 2009.

\bibitem{MR1331981}
{\sc D.~Jerison and C.~E. Kenig}, {\em The inhomogeneous {D}irichlet problem in
  {L}ipschitz domains}, J. Funct. Anal., 130 (1995), pp.~161--219,
  \url{https://doi.org/10.1006/jfan.1995.1067}.

\bibitem{MR4487366}
{\sc S.~Kesavan}, {\em Nonlinear functional analysis: a first course}, vol.~28
  of Texts and Readings in Mathematics, Springer, Singapore; Hindustan Book
  Agency, New Delhi, second~ed., [2022] \copyright 2022,
  \url{https://doi.org/10.1007/978-981-16-6347-5}.

\bibitem{MR2641453}
{\sc A.~Khapalov}, {\em Controllability of partial differential equations
  governed by multiplicative controls}, Springer-Verlag, 2010,
  \url{https://doi.org/10.1007/978-3-642-12413-6}.

\bibitem{MR3941634}
{\sc S.~Kim and G.~Sakellaris}, {\em Green's function for second order elliptic
  equations with singular lower order coefficients}, Comm. Partial Differential
  Equations, 44 (2019), pp.~228--270,
  \url{https://doi.org/10.1080/03605302.2018.1543318}.

\bibitem{MR1786735}
{\sc D.~Kinderlehrer and G.~Stampacchia}, {\em An introduction to variational
  inequalities and their applications}, vol.~31 of Classics in Applied
  Mathematics, Society for Industrial and Applied Mathematics (SIAM),
  Philadelphia, PA, 2000, \url{https://doi.org/10.1137/1.9780898719451}.

\bibitem{MR2536007}
{\sc A.~Kr\"oner and B.~Vexler}, {\em A priori error estimates for elliptic
  optimal control problems with a bilinear state equation}, J. Comput. Appl.
  Math., 230 (2009), pp.~781--802,
  \url{https://doi.org/10.1016/j.cam.2009.01.023}.

\bibitem{ledzewicz2004}
{\sc U.~Ledzewicz and H.~Schattler}, {\em Application of control theory in
  modelling cancer chemotherapy}, Society of Control Robotics Systems: Academic
  Conference Papers,  (2004), pp.~330--335.

\bibitem{MR271512}
{\sc J.-L. Lions}, {\em Optimal control of systems governed by partial
  differential equations}, vol.~Band 170 of Die Grundlehren der mathematischen
  Wissenschaften, Springer-Verlag, New York-Berlin, 1971.
\newblock Translated from the French by S. K. Mitter.

\bibitem{MR161019}
{\sc W.~Littman, G.~Stampacchia, and H.~F. Weinberger}, {\em Regular points for
  elliptic equations with discontinuous coefficients}, Ann. Scuola Norm. Sup.
  Pisa Cl. Sci. (3), 17 (1963), pp.~43--77.

\bibitem{MR2641539}
{\sc V.~Maz'ya and J.~Rossmann}, {\em Elliptic equations in polyhedral
  domains}, vol.~162 of Mathematical Surveys and Monographs, American
  Mathematical Society, Providence, RI, 2010,
  \url{https://doi.org/10.1090/surv/162}.

\bibitem{MR0332249}
{\sc R.~R. Mohler}, {\em Bilinear control processes. {W}ith applications to
  engineering, ecology, and medicine}, Mathematics in Science and Engineering,
  Vol. 106, Academic Press [Harcourt Brace Jovanovich, Publishers], New
  York-London, 1973.

\bibitem{MR3014456}
{\sc T.~Roub\'{\i}\v{c}ek}, {\em Nonlinear partial differential equations with
  applications}, vol.~153 of International Series of Numerical Mathematics,
  Birkh\"auser/Springer Basel AG, Basel, second~ed., 2013,
  \url{https://doi.org/10.1007/978-3-0348-0513-1}.

\bibitem{MR431753}
{\sc A.~H. Schatz and L.~B. Wahlbin}, {\em Interior maximum norm estimates for
  finite element methods}, Math. Comp., 31 (1977), pp.~414--442,
  \url{https://doi.org/10.2307/2006424}, \url{https://doi.org/10.2307/2006424}.

\bibitem{SERHAL2025104362}
{\sc S.~Serhal, G.~Chamoun, M.~Saad, and T.~Sayah}, {\em Bilinear optimal
  control for chemotaxis model: The case of two-sidedly degenerate diffusion
  with volume-filling effect}, Non. An.: Real World Applications, 85 (2025),
  p.~104362, \url{https://doi.org/10.1016/j.nonrwa.2025.104362}.

\bibitem{MR0192177}
{\sc G.~Stampacchia}, {\em Le probl\`eme de {D}irichlet pour les \'equations
  elliptiques du second ordre \`a{} coefficients discontinus}, Ann. Inst.
  Fourier (Grenoble), 15 (1965), pp.~189--258,
  \url{http://www.numdam.org/item?id=AIF_1965__15_1_189_0}.

\bibitem{MR3169756}
{\sc J.~L. Taylor, S.~Kim, and R.~M. Brown}, {\em The {G}reen function for
  elliptic systems in two dimensions}, Comm. Partial Differential Equations, 38
  (2013), pp.~1574--1600, \url{https://doi.org/10.1080/03605302.2013.814668}.

\bibitem{MR2583281}
{\sc F.~Tr\"{o}ltzsch}, {\em Optimal control of partial differential
  equations}, vol.~112 of Graduate Studies in Mathematics, American
  Mathematical Society, Providence, RI, 2010,
  \url{https://doi.org/10.1090/gsm/112}.

\bibitem{MR2839219}
{\sc M.~Ulbrich}, {\em Semismooth {N}ewton methods for variational inequalities
  and constrained optimization problems in function spaces}, SIAM, 2011,
  \url{https://doi.org/10.1137/1.9781611970692}.

\bibitem{Zerrik_etal}
{\sc E.~H. Zerrik and N.~El~Boukhari}, {\em Constrained bilinear control
  problem: Application to a cancer chemotherapy model}, International Journal
  of Biomathematics, 10 (2017), p.~1750054,
  \url{https://doi.org/10.1142/S1793524517500541}.

\end{thebibliography}

\end{document}